\documentclass[11pt,a4paper]{amsart}
\usepackage[margin=1.1in]{geometry}

\usepackage{amsmath,amssymb,amscd}
\usepackage{graphicx}
\usepackage{mathtools}
\usepackage{array,longtable}
\usepackage[all]{xy}
\usepackage{trimclip} 
\usepackage{hyperref}
\usepackage{extarrows}

\usepackage{xcolor}
\hypersetup{
	colorlinks,
	linkcolor={red!50!black},
	citecolor={blue!50!black},
	urlcolor={blue!80!black}
}

\usepackage{accents}

\usepackage[backend=bibtex,style=alphabetic]{biblatex}
\DeclareFieldFormat{extraalpha}{#1}

\DeclareLabelalphaTemplate{
	\labelelement{
		\field[final]{shorthand}
		\field{label}
		\field[strwidth=3,strside=left,ifnames=1]{labelname}
		\field[strwidth=1,strside=left]{labelname}
	}
}

\nocite{*}

\newtheorem{theorem}{Theorem}[section]
\newtheorem{lemma}[theorem]{Lemma}
\newtheorem{corollary}[theorem]{Corollary}
\newtheorem{proposition}[theorem]{Proposition}

\theoremstyle{definition}
\newtheorem{definition}[theorem]{Definition}
\newtheorem{example}[theorem]{Example}
\newtheorem{conjecture}[theorem]{Conjecture}
\newtheorem{algorithm}[theorem]{Algorithm}

\theoremstyle{remark}
\newtheorem{remark}[theorem]{Remark}
\newtheorem{convention}[theorem]{Convention}

\newtheorem{question}[theorem]{Question}

\numberwithin{equation}{section}

\def \k{{\bf k}}

\DeclareMathOperator{\Aut}{Aut}
\DeclareMathOperator{\coker}{coker}
\DeclareMathOperator{\E}{E}
\DeclareMathOperator{\C}{\mathtt{C}}
\DeclareMathOperator{\End}{End}
\DeclareMathOperator{\Ext}{Ext}
\DeclareMathOperator{\ext}{ext}

\DeclareMathOperator{\Hom}{Hom}

\DeclareMathOperator{\Img}{Im}
\DeclareMathOperator{\ind}{ind}
\DeclareMathOperator{\Ind}{Ind}

\DeclareMathOperator{\Ker}{Ker}

\DeclareMathOperator{\PHom}{PHom}
\DeclareMathOperator{\IHom}{IHom}

\DeclareMathOperator{\rep}{rep}

\DeclareMathOperator{\st}{st}

\DeclareMathOperator{\sgn}{sgn}

\DeclareMathOperator{\dv}{\underline{\dim}}

\DeclareMathOperator{\Gen}{{Gen}}

\DeclareMathOperator{\hatj}{\hat{\jmath}}

\newcommand{\op}[1]{\operatorname{#1}}

\newcommand{\Gr}{\operatorname{Gr}}
\newcommand{\mb}[1]{\mathbb{#1}}
\newcommand{\mc}[1]{\mathcal{#1}}
\newcommand{\mf}[1]{\mathfrak{#1}}

\renewcommand{\b}[1]{\bold{#1}}
\newcommand{\e}{\op{e}}
\newcommand{\eu}{{\sf e}}
\newcommand{\g}{{\sf g}}
\newcommand{\N}{{\sf N}}
\renewcommand{\S}{{\mc{S}}}

\newcommand{\ep}{{\epsilon}}

\newcommand{\proj}{\operatorname{proj}\text{-}}

\newcommand{\br}[1]{\overline{#1}}
\newcommand{\innerprod}[1]{\langle#1\rangle}

\newcommand{\sm}[1]{\left(\begin{smallmatrix}#1\end{smallmatrix}\right)}

\newcommand{\wtd}[1]{\widetilde{#1}}
\newcommand{\lwtd}[1]{ \raisebox{-0.26ex}{$\widetilde{\phantom{#1}}$}\mkern-12mu#1 }

\newcommand{\sgnc}{{\check{\sgn}}}

\newcommand{\drep}{{rep}}
\newcommand{\CQ}{\mc{C}_{Q,\S}}
\newcommand{\eQS}{{_e(Q,\S)}}

\newcommand{\dc}{{\check{d}}}

\newcommand{\ec}{{\check{e}}}
\newcommand{\fc}{{\check{f}}}
\newcommand{\tc}{{\check{t}}}

\newcommand{\Nc}{{\check{\sf N}}}
\newcommand{\dtc}{{\check{\delta}}}

\newcommand{\epc}{{\check{\ep}}}
\newcommand{\Ec}{{\check{\E}}}

\newcommand{\M}{\mc{M}}
\renewcommand{\L}{\mc{L}}

\newcommand{\mub}{\mu_{\b{k}}}

\newcommand{\Rcirc}{R_{\circlearrowleft}}

\makeatletter
\newcommand{\lperp}{\mathrel{\mathpalette\@lperp\relax}}
\newcommand{\@lperp}[2]{%
	\clipbox{{.47\width} 0pt 0pt 0pt}{$\m@th#1\perp$}%
}
\newcommand{\rperp}{\mathrel{\mathpalette\@rperp\relax}}
\newcommand{\@rperp}[2]{%
	\clipbox{0pt 0pt {.51\width} 0pt}{$\m@th#1\perp$}
}
\makeatother

\newcommand{\eo}[1]{{_{#1}\!\rperp\,}}
\newcommand{\eot}[1]{{_{#1}\lwtd{\rperp}\,}}
\newcommand{\eor}[1]{{\lperp_{#1}}}
\newcommand{\eort}[1]{{\,\lwtd{\lperp}_{#1}}}
\newcommand{\heo}[1]{{_{#1}\!\perp}}

\newcommand{\comp}[1]{{{\scriptstyle \lfloor}{#1} {\scriptstyle\rfloor}}}%

\newcommand{\Markovext}{
	\vcenter{\xymatrix@R=7ex@C=7ex{
			3 \ar@<0.3ex>[r] \ar@<-0.3ex>[r] \ar[d] & 4 \ar@<-0.75ex>[d] \ar@<-0.25ex>[d]\ar@<0.25ex>[d]\ar@<0.75ex>[d]   \\
			1 \ar@<0.3ex>[r] \ar@<-0.3ex>[r] & 2 \ar@<0.5ex>[ul] \ar[ul] \ar@<-0.5ex>[ul] \\
		}
}}

\begin{document}

\title{On the Orthogonal Projections}
\author{Jiarui Fei}
\address{School of Mathematical Sciences, Shanghai Jiao Tong University}
\email{jiarui@sjtu.edu.cn}
\thanks{The author was supported in part by National Natural Science Foundation of China (No. 12131015 and No. 12571038)}

\date{}
\dedicatory{}

\begin{abstract}
Let $A$ be a finite-dimensional algebra and $e$ a rigid presentation. We give an explicit formula for the left adjoint to the inclusion $\rep(e^\perp)\hookrightarrow\rep(A)$ and construct a bijection between presentations in $\rep(e^\perp)$ and presentations compatible with $e$. The induced bijection on weights preserves generic $\E$-invariants and canonical decompositions. 
If $A$ is Jacobian, the reduction algebra is Jacobian for every rigid $e$. We prove mutation formulas for the signed completions and for the dimension vectors of the simple objects of $\rep(e^\perp)$, enabling an algorithm to find the projected quiver with potential when $e$ is extended-reachable. 
Additionally, we introduce a modified projection for quivers with potentials that preserves general presentations.
For applications to cluster algebras, we establish a connection to the stabilization functors.
\end{abstract}

\maketitle

\section*{Introduction}
\subsection{Perpendicular categories and reduction}
Geigle and Lenzing studied perpendicular, or orthogonal, categories in \cite{GL}. One of their principal examples is
\[ E^{\perp_1}:=\{M\in\rep(A)\mid \Hom(E,M)=\Ext^1(E,M)=0\}, \]
where $E\in\rep(A)$ is $\Ext^1$-rigid. Under suitable hypotheses, the inclusion $E^{\perp_1}\hookrightarrow\rep(A)$ has a left adjoint
\[ l_E:\rep(A)\longrightarrow E^{\perp_1}, \]
which may be viewed as projection onto the perpendicular category.

The $\E$-invariant for presentations over an arbitrary finite-dimensional algebra was introduced in \cite{DF}, following the work of Derksen, Weyman, and Zelevinsky. For a rigid projective presentation $e$ of weight $\ep$, replacing the $\Ext^1$-condition by $\E$-orthogonality gives
\begin{equation}\label{eq:eperp}
\rep(e^\perp)=\{M\in\rep(A)\mid \Hom(e,M)=\E(e,M)=0\}.
\end{equation}
As explained in \cite{DWr} and \cite{Fm2,Ftf}, this is the category of $\ep$-semistable representations in the sense of King \cite{Ki}. Jasso proved that $\rep(e^\perp)$ is equivalent to the module category of an algebra obtained from the Bongartz completion of $e$ \cite{Ja}. More generally, Angeleri H\"ugel, Marks, and Vit\'oria proved that the perpendicular subcategory associated with a partial silting module is bireflective \cite[Proposition 3.3]{AHMV16ref}. In the finite-dimensional setting, both adjoints preserve finite-dimensional modules.

Schofield studied $E^{\perp_1}$ for hereditary algebras and applied perpendicular categories to semi-invariants of quiver representations \cite{S1}. The constructions in Sections \ref{S:OP} and \ref{S:ereg} appeared in the unpublished manuscript \cite{Fm2} and were later used in \cite{Fs2}; they are also used in \cite{Fsj}.

\subsection{Projections from rigid presentations}
Let $e$ be a rigid presentation of weight $\ep$. Recall the right $\E$-truncation defined by the canonical triangle
\[ h_1e[-1]\xrightarrow{\rm can}d\longrightarrow\eot e(d)\longrightarrow h_1e. \]
For a module $M$ with minimal presentation $d_M$, define
\[ L_e(M):=f_\ep\!\left(\coker\eot e(d_M)\right), \]
where $f_\ep$ is the torsion-free functor associated with $\ep$. Sections \ref{ss:PNC} and \ref{ss:TF} recall the truncation and the torsion pair. The following theorem identifies this formula with the left adjoint supplied by silting reduction.

\begin{theorem}[Theorem \ref{T:L_e}]\label{intro:T:L_e}
The functor $L_e:\rep(A)\to\rep(e^\perp)$ is left adjoint to the inclusion $\rep(e^\perp)\hookrightarrow\rep(A)$.
\end{theorem}

Dually, for an injective presentation $\ec$ of weight $\epc$,
\[ R_\ec=\tc_{-\epc}\,\eort\ec:\rep(A)\longrightarrow\rep({}^\perp\ec) \]
is right adjoint to the inclusion.

\begin{corollary}[Corollaries \ref{C:equi} and \ref{C:bimod}]\label{intro:T:adj}
The category $\rep(e^\perp)$ is equivalent to $\rep({_e}A)$, where
\[ {_e}A=\End\!\left(\bigoplus\Ind(L_e(A))\right). \]
Moreover,
\[ L_e(M)\cong\Hom(M,R_{\nu e}(A^*))^*, \qquad R_\ec(M)\cong\Hom(L_{\nu^{-1}\ec}(A),M). \]
\end{corollary}

The positive and negative completions of $e$ determine the indecomposable projectives and simple objects of $\rep(e^\perp)$. Write $A=KQ/I$, $n=|Q_0|$, and $e^\pm=e\oplus e_c^\pm$ for the completions. For an ordered cluster $d$, its $C$-matrix $\C_d$ is the inverse of the matrix whose rows are the weights of its summands.
 The dimension vectors of the simples form the matrix $\C_{e^\perp}$, obtained from the complementary columns of $\C_{e^+}$ or $-\C_{e^-}$, with consistent ordering \cite{Fu,As,Tr}.
The multiplicities $[\eort{e}(P_i[1]):e]$ are related to the Schur-reduction of $e$ (Proposition~\ref{P:Schred}). 
The Schur reductions are defined in Definition~\ref{D:Schredbr}; see also \cite{Frs}.

\subsection{The presentation and weight bijections}
A presentation $d$ is \emph{compatible} with $e$ if
\[ \e(d,e)=\e(e,d)=0. \]
Let $\drep(\comp e)$ be the set of compatible presentations with no direct summand in $\op{add}(e)$. Based on the work of Jasso, Buan and Marsh constructed a bijection between the rigid objects of $\drep(e^\perp)$ and $\drep(\comp e)$ \cite{BM1,BM2}. We will show that this bijection can be extended to all presentations, not just rigid ones.

The induced additive functor $L_e$ on projective presentations may produce additional shifted-projective summands. 
We define a normalized map $\L_e$ on isomorphism classes, preserving the negative projectives corresponding to negative complements (see Definition \ref{D:Lext}). A presentation $d$ is \emph{$e$-regular} when
\[ \L_e(d)\simeq L_e(d). \]
Every compatible presentation is $e$-regular. Proposition \ref{P:induced-general} establishes genericity and rigidity preservation for the induced functor $L_e$ on presentations with $\E(d,e)=0$; $e$-regularity transfers these conclusions to $\L_e$ and determines its weight.

\begin{theorem}[Theorems \ref{T:bijection} and \ref{T:bijection_space}]
\label{intro:T:bijection}
Assume that $e$ is indecomposable. The maps
\[ \L_e:\drep(\comp e)\longrightarrow\drep(e^\perp) \qquad\text{and}\qquad \mc I_e:\drep(e^\perp)\longrightarrow\drep(\comp e) \]
are mutually inverse bijections on isomorphism classes, and they preserve and reflect rigidity.

If $\e(\delta,e)=0$, then $\L_e$ sends a general presentation of weight $\delta$ to a general presentation in $\drep(e^\perp)$. If $\delta$ is $e$-regular, the image has weight $\delta\C_{e^\perp}$.
\end{theorem}

Remark \ref{r:decomposable} extends these statements to arbitrary basic rigid presentations. The same projection induces a bijection of weights. Let $e=\bigoplus_{a=1}^r e_a$ be basic rigid, let $\mc W_e$ be the set of weights whose general presentations lie in $\drep(\comp e)$, and define
\[ p_e:\mc W_e\longrightarrow\mb Z^{n-r}, \qquad p_e(\delta)=\delta\C_{e^\perp}. \]

\begin{theorem}[Theorem \ref{T:CDprojection}]
\label{intro:T:CDprojection}
The map $p_e$ is a bijection. For $\delta,\eta\in\mc W_e$,
\[ \e_A(\delta,\eta) = \e_{e^\perp}\bigl(p_e(\delta),p_e(\eta)\bigr). \]
Both $p_e$ and its inverse preserve canonical decompositions. Hence they induce mutually inverse bijections between the indecomposable weights in $\mc W_e$ and the indecomposable weights of $\rep(e^\perp)$.
\end{theorem}

\subsection{Quivers with potentials}
A natural goal, suggested by Schofield's induction, is to develop perpendicular reduction within the representation theory of quivers with potentials (QPs). Here $A=J(Q,\S)$ is assumed finite-dimensional. The first question is whether $\rep(e^\perp)$ is itself the module category of a Jacobian algebra. For an effective induction, one would also like to construct the projected QP directly from $(Q,\S)$. We obtain two answers: Keller--Liu gives a noncanonical QP realization for every rigid $e$, while an extended-reachable completion permits an explicit construction by mutation and restriction. If $e^\varepsilon$ is extended-reachable, choose a sequence of mutations and global shifts taking it to a signed projective cluster; denote the resulting QP by $\mu_{e^\varepsilon}(Q,\S)$.

\begin{theorem}[Corollary \ref{C:perpQP-abstract} and Theorem \ref{T:perpQP}]\label{intro:T:perpQP}
For every rigid presentation $e$, there is a Jacobi-finite complete QP $(Q_e,\S_e)$, possibly with loops or $2$-cycles, such that
\[ \rep(Q_e,\S_e)\simeq\rep(e^\perp). \]
This QP need not be canonical. If the completion $e^\varepsilon=e\oplus e_c^\varepsilon$ is extended-reachable and $I_\varepsilon$ denotes the vertices corresponding to $e_c^\varepsilon$, then one may take
\[ {_e(Q,\S)}^\varepsilon = \mu_{e^\varepsilon}(Q,\S)\big|_{I_\varepsilon}. \]
If both completions are extended-reachable, the two resulting Jacobian algebras are isomorphic. Rigidity and nondegeneracy of $(Q,\S)$ pass to these projected QPs.
\end{theorem}

The explicit construction requires a positive or negative complement of $e$: its completion determines which vertices are deleted after mutation. The next theorem makes these complements computable along mutation sequences. 
Here $\sigma_j$ exchanges its $j$-th summand at a fixed QP.
Let $e=\bigoplus_{a=1}^r e_a$ be basic rigid, put $\ep=\sum_a\ep_a$, and write $e^\varepsilon=e\oplus e_c^\varepsilon$ for $\varepsilon\in\{+,-\}$. For the following formulas, suppose that $e^\varepsilon$ is extended-reachable.

\begin{theorem}[Theorem \ref{T:mucomp}]\label{intro:T:mucomp}
If $\ep(k)=0$, let $j$ be the unique complementary index for which the $j$-th column of $\varepsilon\C_{e^\varepsilon}$ is $\eu_k$. Then
\[
(\mu_ke)^\varepsilon=
\begin{cases}
\mu_k(e^\varepsilon),&\ep(k)\neq0,\\
\sigma_j\mu_k(e^\varepsilon),&\ep(k)=0.
\end{cases}
\]
Thus an additional exchange occurs precisely when $\ep(k)=0$.
\end{theorem}

If only the projected quiver is needed, there is a simpler approach. The dimension vectors of the simple objects of $\rep(e^\perp)$ form the complementary columns of the completion $C$-matrices. Their piecewise-linear mutation rule avoids explicitly performing the additional exchange in Theorem~\ref{intro:T:mucomp}. Let $B(Q)$ be the exchange matrix of $Q$ and $b_k$ its $k$-th row. Write $[a]_+=\max(a,0)$ entrywise and $\max_+=\max$, $\max_-=\min$.
\begin{theorem}[Theorem \ref{T:musimple}]\label{intro:T:musimple}
Fix a sign $\varepsilon$ for which $e^\varepsilon$ is extended-reachable. Write $\gamma_i$ and $\gamma_i'$ for the $i$-th columns of $\C_{e^\varepsilon}$ and $\C_{(\mu_ke)^\varepsilon}$, respectively. These matrices differ only in row $k$, and
\[
\gamma_i'(k)=
\begin{cases}
[\sgn(\ep(k))b_k]_+\gamma_i-\gamma_i(k),&\ep(k)\neq0,\\
\max_{\mp}\!\bigl([-b_k]_+\gamma_i,[b_k]_+\gamma_i\bigr)-\gamma_i(k),&\ep(k)=0\text{ and }i\neq j,\\
\pm1,&\ep(k)=0\text{ and }i=j.
\end{cases}
\]
The upper and lower signs correspond to $\varepsilon=+$ and $\varepsilon=-$.
\end{theorem}

These formulas lead to two complementary algorithms. Given a reachability sequence, Algorithm~\ref{a:Qep} computes the projected QP, including its potential. Algorithm~\ref{a:Qepsimple} needs only the simple dimension-vector matrix $\C_{e^\perp}$ to find the projected exchange matrix
\[ B_e=\C_{e^\perp}^{\top}B(Q)\C_{e^\perp}. \]
For indecomposable $e$, the final columns also give the Schur-reduction vectors $\gamma_\pm(e)$, whose mutation rule follows from the same calculation.

There is a second obstacle to Schofield-style induction: the normalized projection $\L_e$ does not, in general, send arbitrary general presentations to general ones. Calabi--Yau reduction gives two mutation-equivariant $\E$-orthogonalizations
\[ \sqcup_e^\pm:\drep(Q,\S)\longrightarrow\drep(\comp e), \]
and hence modified projections $\L_e^\pm=\L_e\circ\sqcup_e^\pm$. Once a projected QP and its reduction equivalence are chosen, the equivalence, including its action on negative projectives, gives a bijection
\[\lambda_e^\varepsilon:\drep(e^\perp)\xrightarrow{\sim}\drep({_e(Q,\S)}^\varepsilon).\]
The corresponding projection is $\L_{(e)}^\varepsilon = \lambda_e^\varepsilon\circ\L_e^\varepsilon$.

The positive $\E$-orthogonalization, followed by projection, recovers the stabilization functor defined by torsion pairs in \cite{Fcf}. Write $\heo\ep$ for this functor and $\perp_{\epc}$ for its dual.
\begin{theorem}[Theorem \ref{T:Lstab}]\label{intro:T:Lstab}
The assignments $L_e^+$ and $R_\ec^+$ are functors, and
\[ L_e^+\cong\heo\ep, \qquad R_\ec^+\cong\perp_{\epc}. \]
\end{theorem}

When a completion is extended-reachable, this categorical construction has an especially concrete form.
\begin{theorem}[Theorem \ref{T:L_e^-}]\label{intro:T:L_e^-}
Fix $\varepsilon\in\{+,-\}$ and suppose that $e^\varepsilon$ is extended-reachable. Choose a mutation sequence taking $e^\varepsilon$ to a signed projective cluster. Then $\L_{(e)}^\varepsilon(d)$ is obtained by applying this sequence to $d$ and restricting to the complementary vertices. This map sends general presentations to general presentations.
\end{theorem}
Example~\ref{E:stabnotgeneric} shows that the preservation of general presentations does not extend to all
	finite-dimensional algebras.
Corollary~\ref{C:mulift} states that mutation sequences can be pushed down and lifted, 
	which can be useful in certain induction procedure.
\[
\xymatrix{
\drep(Q,\S)\ar[r]^{\mub}\ar[d]^{\L_{(e)}^\varepsilon}
&\drep(Q,\S)'\ar[d]^{\L_{(e')}^\varepsilon}\\
\drep({_e(Q,\S)}^\varepsilon)\ar[r]^{\mub^{e,\varepsilon}}
&\drep({_{e'}(Q,\S)'}^\varepsilon)
}
\]

Another interesting corollary of Theorems 1.8 and 1.9 is related to the cluster algebra theory.
It roughly says the following.
\begin{corollary}[Corollary \ref{C:faces}]\label{intro:C:faces}
Let $F_M$ be the $F$-polynomial of a general representation of $(Q,\S)$, let ${\sf\Lambda}$ be a facet of its Newton polytope, and suppose that the indivisible outer normal $\ep$ is the weight of a rigid presentation $e$. Up to a monomial factor and a monomial change of variables, the restriction of $F_M$ to ${\sf\Lambda}$ is the $F$-polynomial of the stabilization $\heo\ep(M)$. If $e^+$ is extended-reachable, this stabilization is general over the projected QP ${_e(Q,\S)}^+$.
\end{corollary}
The rigidity hypothesis on the facet normal cannot simply be removed:
Example~\ref{E:tame-normal-face} gives a tame-normal counterexample, even
at the level of Newton polytopes and for generic potentials.

\subsection{Organization}
Section \ref{S:Prelim} recalls projective presentations, $\E$-truncations, and the associated torsion pairs. Section \ref{S:OP} constructs $L_e$, proves the adjunction, and identifies the algebra ${_e}A$. Section \ref{S:SS} describes the simple objects of $\rep(e^\perp)$ through completion matrices and studies Schur reductions. Section \ref{S:ereg} extends $L_e$ to presentations and defines $e$-regularity.

Section \ref{S:bijection} proves the presentation bijection, genericity and rigidity preservation, and the weight bijection preserving generic $\E$-invariants and canonical decompositions. Remark \ref{r:decomposable} gives the extension from indecomposable to arbitrary basic rigid presentations.

From Section \ref{S:QP} onward, $A$ is Jacobian. Section \ref{S:QP} proves the Jacobian realizations and computes the exchange matrix of the reduction. Section \ref{S:inmu} proves the mutation formulas and gives the two algorithms for projected QPs. Section \ref{S:OPmod} constructs the $\E$-orthogonalizations and modified projections, identifies the positive module functor with stabilization, and proves the mutation--restriction description, mutation lifting, and the facet formula for generic $F$-polynomials.

\subsection{Notation}
Let $K$ be an algebraically closed field of characteristic zero, and let $A\cong KQ/I$ be a basic finite-dimensional $K$-algebra. All modules are right modules, all representations are finite-dimensional, and vectors are row vectors unless stated otherwise. The symbols $\delta$ and $\ep$ denote the weights of presentations $d$ and $e$, respectively; $\gamma$ denotes a dimension vector arising from a reduction.

The direct sum of $n$ copies of $M$ is denoted by $nM$. We write $\hom$, $\ext$, and $\e$ for $\dim\Hom$, $\dim\Ext$, and $\dim\E$, respectively, and $(-)^*$ for the $K$-dual. Decorated representations are identified with presentations as explained in Subsection \ref{ss:presentations}.
{\small
\renewcommand{\arraystretch}{1.12}
\setlength{\tabcolsep}{4pt}
\begin{longtable}{@{}>{\raggedright\arraybackslash}p{0.2\textwidth}>{\raggedright\arraybackslash}p{0.52\textwidth}>{\raggedright\arraybackslash}p{0.24\textwidth}@{}}
\textbf{Notation} & \textbf{Meaning} & \textbf{First use}\\
\hline
\endfirsthead
\textbf{Notation} & \textbf{Meaning} & \textbf{First use}\\
\hline
\endhead
$\rep(A)$ & finite-dimensional $A$-modules & \\
$\drep(A)$ & decorated modules / presentations & Subsection \ref{ss:presentations}\\
$S_i$ & simple at vertex $i$ & \\
$P_i,I_i$ & projective cover / injective hull of $S_i$ & \\
$P[1],-P$ & negative presentation $P\to0$ & Subsection \ref{ss:presentations}\\
$\delta_{\M},\dtc_{\M}$ & $\delta$- and $\dtc$-vectors of $\M$ & before Definition \ref{D:HomE}\\
$\E(\M,\mc N),\Ec(\M,\mc N)$ & $\E$-invariant and its dual & Definition \ref{D:HomE}\\
$\e_A,\e_{e^\perp}$ & generic $\E$-values (ambient / reduced) & Theorem \ref{T:CDprojection}\\
${\sf e}_i$ & unit vector in coordinate $i$ & \\
$e,\ E=\coker(e)$ & rigid presentation and its cokernel & \\
$\Ind(e)$ & indecomposable summands of $e$ & \\
$\rep(e^{\lperp}),\rep(e^{\rperp\,})$ & $\hom(e,-)=0$ / $\e(e,-)=0$ subcategories & \eqref{eq:F}, \eqref{eq:Tc}\\
$\rep(e^\perp)$ & $\Hom(e,-)=\E(e,-)=0$ category & \eqref{eq:eperp}\\
$\rep({}^\perp\ec)$ & dual perpendicular category for $\ec$ & after Theorem \ref{T:L_e}\\
$\rep(\comp e)$ & modules with $e$-compatible minimal presentations & before Lemma \ref{L:0fibre}\\
$\drep(\wtd{\comp e})$ & compatible presentations, allowing $e$-summands & Subsection \ref{ss:presentation-bijection}\\
$\drep(\comp e)$ & compatible presentations without $e$-summands & Subsection \ref{ss:presentation-bijection}\\
$L_e,R_\ec$ & left / right perpendicular projections & after \eqref{eq:tf}\\
$\L_e,\mc R_\ec$ & modified projection maps on presentations & Definition \ref{D:Lext}\\
$\mc I_e,\check{\mc I}_\ec$ & inverse lifts on compatible presentations & before Theorem \ref{T:bijection}\\
$\eot e,\eort e$ & right / left $\E$-truncations & Subsection \ref{ss:PNC}\\
$\eo e,\eor e$ & reduced $\E$-truncations & Subsection \ref{ss:PNC}\\
$t_\ep,f_\ep,\tc_\ep,\fc_\ep$ & torsion-pair functors for $\ep$ & Subsection \ref{ss:TF}\\
$\heo\ep$ & stabilization functor for $\ep$ & after \eqref{eq:Fcontain}\\
$e_c^\pm$ & basic positive / negative complements & after \eqref{eq:e0r}\\
$e^\pm=e\oplus e_c^\pm$ & positive / negative completions & Subsection \ref{ss:simple}\\
${_eA}$ & basic algebra for $\rep(e^\perp)$ & Corollary \ref{C:equi}\\
$\Delta_{\br e},\C_{\br e}$ & cluster weight matrix and its inverse & Definition \ref{D:Delta}\\
$\C_{e^\perp}$ & matrix of reduced simple dimension vectors & before Lemma \ref{L:Cmat}\\
$\mc W_e$ & weights with general $e$-compatible presentations & Subsection \ref{ss:weight-bijection}\\
$p_e,\ell_e$ & weight projection and its inverse & Theorem \ref{T:CDprojection}\\
$B(Q)$ & exchange matrix of $Q$ & Section \ref{S:QP}\\
$B_e=\C_{e^\perp}^{\top}B(Q)\C_{e^\perp}$ & reduced exchange matrix & Proposition \ref{P:Breduction}\\
$\mu_k,\sigma_j$ & vertex-$k$ mutation / summand-$j$ exchange & after Theorem \ref{T:+-}; Section \ref{S:inmu}\\
${_e(Q,\S)}^\varepsilon$ & projected QP from $\varepsilon$-completion & before Theorem \ref{T:perpQP}\\
$\wtd{\sqcup}_e^\pm,\sqcup_e^\pm$ & unreduced / reduced $\E$-orthogonalizations & Definition \ref{D:sqproj}\\
$\L_e^\pm,\mc R_\ec^\pm$ & $\E$-orthogonalization projections & Definition \ref{D:Lpm}\\
$\lambda_e^\varepsilon$ & bijection from Calabi--Yau reduction & before Definition \ref{D:Lpm}\\
$\L_{(e)}^\varepsilon$ & projection to a fixed projected QP & Definition \ref{D:Lpm}\\
$F_M$ & ordinary $F$-polynomial of $M$ & Subsection \ref{ss:Fpolynomials}\\
$\N(M),\Nc(M)$ & submodule / quotient Newton polytopes & before Question \ref{q:c1}\\
$\rho_{\ep_{\hatj}-\st}(e),\rho_\pm(e)$ & stable / Schur reductions of $e$ & Definitions \ref{D:Schred} and \ref{D:Schredbr}\\
$\gamma_\pm(e)$ & signed Schur-reduction vectors & Definition \ref{D:Schredbr}
\end{longtable}
}

\noindent\textbf{Sign convention.}
In any formula containing $\pm$ or $\mp$, the upper signs and the lower signs are taken consistently. Thus $+$ refers to the positive complement or reduction and $-$ to the negative one.

\section{Preliminaries} \label{S:Prelim}
\subsection{Projective presentations}\label{ss:presentations}
Following \cite{DF}, a homomorphism between projective representations is called a {\em projective presentation}, or simply a presentation. 
As a full subcategory of the category of complexes in $\rep A$, the category $K^{[-1,0]}(\proj A)$ of projective presentations is also Krull--Schmidt.
A presentation $P_-\to P_+$ will also be viewed as an object of the homotopy category $K^b(\proj A)$, with $P_-$ in degree $-1$ and $P_+$ in degree $0$.

A presentation $d$ is called {\em negative} if $P_+=0$ but $P_-\neq 0$;
is called {\em neutral} if $P_-=P_+$ and $d$ is the identity map;
is called {\em minimal} if $d$ is a minimal presentation of $\coker(d)$.

\begin{lemma}\label{L:decmnn} Every presentation is a direct sum of a minimal, a negative, and a neutral presentation. Up to homotopy equivalence, the neutral summand may be omitted.
\end{lemma}


Let $\drep(A)$ denote the set of isomorphism classes of decorated
representations $\M=(M,M^-)$ of $A$. The assignment
\[ \M\longmapsto d_{\M}:=d_M\oplus P(M^-)[1] \]
is a bijection from $\drep(A)$ to the set of isomorphism classes in
$K^{[-1,0]}(\proj A)$. Here $d_M$ is a minimal projective presentation
of $M$, and
\[ P(M^-)=\bigoplus_{u\in Q_0}\dim M^-(u)P_u. \]
The negative simple $S_u^-$ corresponds to $P_u[1]=(P_u\to0)$.

Let $\nu$ be the Nakayama functor $\Hom(-,A)^*$.
There is a map still denoted by $\nu$ sending a projective presentation to an injective one
\[ P_-\to P_+\ \mapsto\ \nu(P_-) \to \nu(P_+). \]
If $d$ has no direct summand of the form $P_i\to0$, then $\ker(\nu d)=\tau\coker(d)$, where $\tau$ is the classical Auslander--Reiten translation. 
The classical Auslander--Reiten translation extends to decorated representations as in \cite{DF}:
\begin{equation}\label{eq:taunu} \xymatrix{\M \ar[r]\ar@{<->}[d] & \tau \M \ar@{<->}[d] \\ d_{\M} \ar[r] & \nu(d_{\M})} \end{equation}
\begin{convention}
\begin{enumerate}
	\item Whenever no ambiguity arises, identify $M$ with $(M,0)$ and with its minimal presentation $d_M$, and identify a decorated representation $\M$ with $d_{\M}$.
	\item For a presentation $e$ with $E=\coker(e)$, write
	\[ \tau E:=\coker(\tau e)=\ker(\nu e). \]
	When $e$ is the minimal projective presentation of $E$, this is the classical Auslander--Reiten translate. Otherwise, $\tau E$ is understood relative to the specified presentation $e$.
	\item The negative presentation $P\to0$ is denoted by $P[1]$ or $-P$. The presentation $0\to P$ is denoted by $P[0]$ or $+P$; the symbol $[0]$ emphasizes that it is regarded as a presentation.
\end{enumerate}
\end{convention}

We denote by $P_u$ (resp. $I_u$) the indecomposable projective (resp. injective) representation of $A$ corresponding to the vertex $u$ of $Q$. 
For $\beta \in \mb{Z}_{\geq 0}^{Q_0}$ we write $P(\beta)$ for $\bigoplus_{u\in Q_0} \beta(u)P_u$. 
The {\em $\delta$-vector} \footnote{The $\delta$-vector is the same one defined in \cite{DF}, but is the negative of the $\g$-vector defined in \cite{DWZ2}. }
(or {\em weight vector}) of a presentation 
$d: P(\beta_-)\to P(\beta_+)$
is the difference $\beta_+-\beta_- \in \mb{Z}^{Q_0}$.
The $\delta$-vector represents the corresponding class in the Grothendieck group of $K^b(\proj A)$.

\begin{definition}[{\cite{DWZ2,DF}}] \label{D:HomE} Given any projective presentation $d: P_-\to P_+$ and any $N\in \rep(A)$, we define $\Hom(d,N)$ and $\E(d,N)$ to be the kernel and cokernel of the induced map:
	\begin{equation}\label{eq:HE} 0\to \Hom(d,N)\to \Hom(P_+,N) \xrightarrow{} \Hom(P_-,N) \to \E(d, N)\to 0. \end{equation}
It follows from \eqref{eq:HE} that 
\begin{equation}\label{eq:h-e} \hom(d,N)- \e(d, N) = \delta\cdot \dv N. \end{equation}
	Similarly for an injective presentation $\dc: I_+\to I_-$, we define $\Hom(M,\dc)$ and $\Ec(M,\dc)$ to be the kernel and cokernel of the induced map $\Hom(M,I_+) \xrightarrow{} \Hom(M,I_-)$.
	By definition, 
	\[ \Hom(d,N) = \Hom(\coker(d),N)\ \text{ and }\ \Hom(M,\dc) = \Hom(M,\ker(\dc)). \]
	We set $\Hom(\M,\mc{N}):=\Hom(d_{\M},N)=\Hom(M,\dc_{\mc{N}})$,\  
	$\E(\M,\mc{N}) := \E(d_{\M},N)$ and $\Ec(\M,\mc{N}) := \Ec(M,\dc_{\mc{N}})$.
\end{definition}	
\noindent By definition, $\Hom(\M,\mc{N})=\Hom(M,N)$.\footnote{This definition is slightly different from the one in \cite{DWZ2}, which involves the decorated part.}
Set $\E(d_{\M},d_{\mc N})=\E(\M,\mc N)$ and $\Ec(\dc_{\M},\dc_{\mc N})=\Ec(\M,\mc N)$. For their interpretation in terms of $d_{\M}$ and $d_{\mc N}$, see \cite{DF}. The following identities hold:
\begin{equation}\label{eq:H2E} \E(\M,-)=\Hom(-,\tau\M)^* \text{ and }\ \check{\E}(-,\M)=\Hom(\tau^{-1}\M,-)^*. \end{equation}


\subsection{General presentations}
By a {\em general presentation} in $\Hom(P_-,P_+)$, we mean a presentation in some open (and thus dense) subset of $\Hom(P_-,P_+)$.
Any $\delta\in \mb{Z}^{Q_0}$ can be written as $\delta = \delta_+ - \delta_-$ where $\delta_+=\max(\delta,0)$ and $\delta_- = \max(-\delta,0)$. We put 
\[ \PHom(\delta):=\Hom(P(\delta_-),P(\delta_+)). \]
For $\beta_-,\beta_+\in\mb Z_{\geq0}^{Q_0}$, a general presentation in $\Hom(P(\beta_-),P(\beta_+))$ is homotopy equivalent to a general presentation in $\PHom(\beta_+-\beta_-)$. 

There is some open subset $U$ of $\PHom(\delta)$ such that for any $d\in U$, $\Hom(d,M)$ has constant dimension for a fixed $M\in \rep(A)$. This implies that $\E(d,M)$ has constant dimension as well,
and that $\coker(d)$ has a constant dimension vector, denoted $\dv(\delta)$.
A general representation of weight $\delta$ is defined as the cokernel of a general presentation in $\PHom(\delta)$. All general representations in this article are of this form.
\begin{definition} \label{D:home} We denote by $\hom(\delta,M)$ and $\e(\delta,M)$ the value of 
	$\hom(d,M)$ and $\e(d,M)$ for a general presentation $d\in \PHom(\delta)$.
	$\hom(M,\dtc)$ and $\check{\e}(M,\dtc)$ are defined analogously.
\end{definition}
\noindent Recall the isomorphism $\Hom(P_i,P_j) \cong \Hom(I_i, I_j) = \Hom(\nu P_i, \nu P_j)$.
If $d$ is general in $\PHom(\delta)$, then $\nu d$ is general in $\IHom(-\delta)$.
Therefore,  
\[ \hom(\delta,M)= \check{\e}(M,-\delta)\ \text{ and }\ \e(\delta,M)= \hom(M,-\delta). \]  

\begin{definition}[\cite{DF}] A weight vector $\delta\in\mathbb{Z}^{Q_0}$ is called {\em indecomposable} if a general presentation in $\PHom(\delta)$ is indecomposable. We call $\delta=\bigoplus_{i=1}^s \delta_i$ a {\em decomposition} of $\delta$ if a general element $d$ in $\PHom(\delta)$ decomposes into $\bigoplus_{i=1}^s d_i$ with each $d_i\in \PHom(\delta_i)$.
It is called the {\em canonical decomposition} of $\delta$ if each $d_i$ is indecomposable.
\end{definition}
\noindent The summands in the canonical decomposition of $\delta$ are {\em sign-coherent}: for each $k\in Q_0$, either $\delta_i(k)\leq0$ for every $i$, or $\delta_i(k)\geq0$ for every $i$.

The function $\dim \E(-,-)$ is upper semi-continuous on $\PHom(\delta_1)\times \PHom(\delta_2)$.
We denote by $\e(\delta_1,\delta_2)$ the minimal value of $\dim \E(-,-)$ on $\PHom(\delta_1)\times \PHom(\delta_2)$.
Canonical decompositions satisfy the following criterion.
\begin{theorem}[{\cite[Theorem 4.4]{DF}}] \label{T:CDPHom} $\delta=\delta_1\oplus \delta_2\oplus\cdots\oplus\delta_s$ is the canonical decomposition of $\delta$ if and only if $\delta_1,\cdots,\delta_s$ are indecomposable, and $\e(\delta_i,\delta_j)=0$ for $i\neq j$.
\end{theorem}


The group $\Aut_A(P_-)\times\Aut_A(P_+)$ acts on $\Hom(P_-,P_+)$ by 
$(g_-,g_+)d =g_+ d g_-^{-1}$.
The space $\E(d,d)$ can be interpreted as the normal space to the orbit of $d$ in $\Hom(P_-, P_+)$.

\subsection{Rigid presentations}
\begin{definition}
	A presentation $d$ is called {\em rigid} 
	if $\E(d,d)=0$ ($\Ec(\dc,\dc)=0$ for an injective presentation $\dc$).
	A representation $M$ is called {\em $\E$-rigid} \footnote{Since $\E(M,M)=\Hom(M,\tau M)^*$, this is also called {\em $\tau$-rigidity} in \cite{AIR}.} (resp. $\Ec$-rigid) if $\E(M,M)=0$ (resp. $\Ec(M,M)=0$).
\end{definition}
\noindent The orbit of such a presentation is thus dense in its ambient space.
Equations \eqref{eq:H2E} and \eqref{eq:taunu} give
$\Ec(M,\dc)\cong\Hom(\coker(\nu^{-1}\dc),M)^*$. Consequently,
\[ \E(d,d)\cong\Hom(\coker(d),\ker(\nu d))^* \cong\Ec(\nu d,\nu d). \]
Thus $d$ is rigid if and only if $\nu d$ is rigid.

Every rigid presentation $d$ admits a {\em maximal rigid} completion $\br d$, characterized by 
$\E(\br{d}\oplus d', \br{d}\oplus d') \neq 0$ for any indecomposable $d'\notin \ind(d)$.
Here we denote by $\ind(d)$ the set of nonisomorphic indecomposable direct summands of $d$.
The maximal rigid presentation can be characterized as follows.
\begin{theorem}[{\cite[Theorem 5.4]{DF}}, \cite{AIR}] \label{T:maxrigid} The following are equivalent for a rigid presentation $d$. \begin{enumerate}
		\item	$d$ is maximal rigid; 
		\item	$|\ind(d)|=|Q_0|$; 
		\item 	$\ind(d)$ generates $K^b(\proj A)$. 
	\end{enumerate}
\end{theorem}

\begin{remark}\label{r:maxrigid} More generally, let $d$ be a maximal rigid presentation and $e$ be any presentation. 
If $\E(e,d)=\E(d,e)=0$, then $e\in \op{add}(d)$ by \cite[Corollary 2.13]{AIR}.
\end{remark}

\begin{definition}
If $d$ is maximal rigid, both $d$ and $\ind(d)$ are called a {\em cluster} of presentations. The weights of the presentations in $\ind(d)$ form the corresponding {\em cluster} of $\delta$-vectors.
A rigid presentation $d$ is called {\em almost complete} if $|\ind(d)|=|Q_0|-1$.
\end{definition}

\begin{theorem}[{\cite[Proposition 5.7]{DF}}, \cite{AIR}] \label{T:+-} An almost complete rigid presentation $d$ has exactly two complements $d_-$ and $d_+$. They are related by the triangle 
\[ d_+\to d'\to h_1d_-\to d_+[1]\ \text{ and }\ h_1d_+\to d''\to d_-\to h_1d_+[1], \]
where $h_1=\dim\E(d_-,d_+)$. Moreover, both $d'\oplus d_-$ and $d''\oplus d_+$ are rigid and $\E(d_+,d_-)=\E(d_+,d')=\E(d'',d_-)=0$. In particular, $h_1=1$ if and only if  $d'=d''$ belongs to the subcategory generated by $\ind(d)$.
\end{theorem}

\begin{definition} We call the above pair $(d_-,d_+)$ an {\em exchange pair} of presentations.
	If $h_1=1$, the exchange pair is called {\em regular}.
	The two clusters $\{d_-\}\cup \ind(d)$ and $\{d_+\}\cup \ind(d)$ are called {\em adjacent} to each other.
If the cluster is ordered and $d_{\pm}$ is the $j$-th element, then $d\oplus d_\mp$ is called the $j$-th exchange of $d\oplus d_\pm$, denoted by $\sigma_j(d\oplus d_\pm)$.
\end{definition}

\subsection{$\E$-truncations and complements} \label{ss:PNC}
Let $e$ be a rigid presentation, not necessarily indecomposable. For
$d\in K^{[-1,0]}(\proj A)$, consider the triangle
\begin{equation}\label{eq:e0} h_1e[-1]\xrightarrow{can}d\longrightarrow\br d\longrightarrow h_1e, \end{equation}
where $h_1=\dim\E(e,d)$ and $can$ is the canonical map. Applying
$\Hom(e,-)$ gives
\[ \Hom(e,h_1e)\xrightarrow{\partial}\Hom(e,d[1])\longrightarrow\Hom(e,\br d[1])\longrightarrow\Hom(e,h_1e[1])=0. \]
The map $\partial$ is surjective by construction, so $\E(e,\br d)=0$.
The assignment $d\mapsto\br d$ is the right {\em $\E$-truncation} $\eot e$.
\begin{remark} \label{r:general}
If $d$ is a general presentation of weight $\delta$, then $\br d$ is general if and only if $\e(\delta, \ep)=0$ or $\e(\ep, \delta)=0$ by \cite[Theorem 3.9]{DF}. 
\end{remark}
This functor has a reduced version $\eo{e}$. By definition $\eo{e}(d)$ is obtained from $\eot{e}(d)$ by removing all summands isomorphic to $e$.
For a representation $M$, let $\br M$ be the cokernel of $\eo e(d_M)$; the resulting assignment is again denoted by $\eo e$.

In \cite{DF}, this truncation is used to construct the {\em positive complement} $\eo{e}(A[0])$ to any rigid presentation $e$. The positive complement can be viewed as an analogue of the Bongartz complement for modules \cite{ASS}.
We denote the direct sum of $\op{Ind}(\eo{e}(A[0]))$ by $e_c^+$, and call it the basic positive complement of $e$.

The left $\E$-truncation $\eort e$ is defined by
\begin{equation}\label{eq:e0r} h_1'e\to \eort{e}(d) \to d \xrightarrow{can} h_1' e[1], \end{equation}
where $h_1'=\dim\E(d, e)$ and $can$ is the canonical map.
It satisfies $\E(\eort e(d),e)=0$; its reduced version is denoted by $\eor e$.
The presentation $\eor{e}(A[1])$ is called the {\em negative complement} of $e$.
We denote the direct sum of $\op{Ind}(\eor{e}(A[1]))$ by $e_c^-$, and its cokernel by $E_c^-$.
Write $e_i^\pm$ for the indecomposable summands of $e_c^\pm$.

For an injective presentation $\ec$, the left and right $\Ec$-truncations and their reduced versions are defined dually:
\begin{align}\notag  &\check{h}_1\check{e}[-1]\xrightarrow{can} \dc \to \eot{\check{e}}(\dc) \to  \check{h}_1\check{e}\\
\label{eq:ec0r}	&\check{h}_1'\check{e}\to \eort{\check{e}}(\dc) \to \dc \xrightarrow{can} \check{h}_1'\check{e}[1]
\end{align}
satisfying
$\Ec(\ec, \eo{\ec}(\dc))=0$ and $\Ec(\eor{\ec}(\dc), \ec)=0$.
We denote the direct sum of $\Ind(\eor{\ec}(A^*[0]))$ and $\Ind(\eo{\ec}(A^*[-1]))$ by $\ec_c^+$ and $\ec_c^-$ respectively.

\begin{lemma}\label{L:tau_comp} $(\nu e)_c^+ \cong \nu (e_c^-)$.
\end{lemma}
\begin{proof} We apply the Nakayama functor to the triangle \eqref{eq:e0r} for $d=A[1]$:
	\[ h_1e \to \eort{e}(A[1]) \to A[1] \xrightarrow{can} h_1e[1] \]
and compare with the triangle \eqref{eq:ec0r} for $\dc=A^*[0]$ (the first row below)
\begin{align*} \xymatrix{ \check{h}_1' \nu e \ar[r]\ar@{=}[d] &  \eort{\nu e}(A^*[0]) \ar[r] \ar@{}[d]|{\rotatebox{90}{\scalebox{2}[1]{$\cong$}}} & A^*[0] \ar[r]^{can} \ar@{=}[d] & \check{h}_1' \nu e[1]  \ar@{=}[d]\\
 \nu h_1 e \ar[r] & \nu (\eort{e}(A[1]))  \ar[r] &  \nu A[1] \ar[r]^{can} & \nu h_1 e[1] }
\end{align*}
Note that $h_1 = \e(A[1], e) = \hom(A,E) = \hom(E,A^*) = \check{\e}(A^*, \nu e) = \check{h}_1'$.
The axioms of a triangulated category give $\eort{\nu e}(A^*[0]) \cong \nu (\eort{e}(A[1]))$.
In particular, they have the same indecomposable summands.
\end{proof}


\subsection{Torsion pairs} \label{ss:TF}
Recall that a {\em torsion class} in $\rep(A)$ is a full subcategory of $\rep(A)$ which is closed under images, direct sums, and extensions, 
and a {\em torsion-free class} in $\rep(A)$ is a full subcategory of $\rep(A)$ which is closed under subrepresentations, direct sums, and extensions.
Consider the torsion free class
\begin{align}\label{eq:F} {\mc{F}}(\delta) = \{N\in \rep(A)\mid \hom(\delta,N) =0 \} := \rep(\delta^{\,\lperp})
\intertext{and the torsion class}
\label{eq:Tc} \check{\mc{T}}(\delta) = \{L\in \rep(A) \mid \e(\delta, L) =0 \}:= \rep(\delta^{\rperp\,}).
\end{align}
We denote their corresponding torsion class and torsion-free class by $\mc{T}(\delta)$ and $\check{\mc{F}}(\delta)$ respectively, and their associated pairs of functors by $(t_\delta,f_\delta)$ and $(\tc_\delta,\fc_\delta)$ (see \cite[Proposition VI.1.4]{ASS}).
For a representation $M$,
$t_\delta(M)$ is the smallest subrepresentation $L$ of $M$ such that $\hom(\delta,M/L)=0$, and
$\tc_\delta(M)$ is the largest subrepresentation $L$ of $M$ such that $\e(\delta,L)=0$ (see \cite[the proof of Proposition VI.1.4]{ASS}).

Only the case of a rigid weight $\delta=\ep$ is used below.
Let $e$ be a general presentation of weight $\ep$ and $E$ be its cokernel.
In this case, the general results in \cite{Fcf, Ftf} specialize to the following containments
\begin{align}\label{eq:T} 	\Gen(E) =	\mc{T}(\ep) &\subseteq \check{\mc{T}}(\ep) = \rep(\ep^{\rperp\,}) \\
	\label{eq:Fcontain} \op{Cogen}(\tau E) = \check{\mc{F}}(\ep) &\subseteq {\mc{F}}(\ep) = \rep(\ep^{\,\lperp}). 
\end{align}
If $\ep$ is maximal rigid, then the above two containments become equalities. 
In \cite{Fcf} we defined the {\em stabilization functor} $\heo{\ep} = \tc_{\ep} f_{\ep} = f_{\ep}\tc_{\ep}$, which is different from the orthogonal projection $L_e$ to be defined in Section \ref{S:OP}.


A representation $M\in \mc{C} \subset \rep(A)$ is called $\E$-projective in $\mc{C}$ if $\E(M,N)=0$ for any $N\in \mc{C}$.
Dually, a representation is $\Ec$-injective in $\mc C$ if the
corresponding dual $\E$-groups vanish.
\begin{lemma}\label{L:Eproj}
One has
\[ \Gen(E)=\Gen(E\oplus E_c^-),\qquad \check{\mc T}(\ep)=\check{\mc T}(\ep\oplus\ep_c^+). \]
The $\E$-projective representations in these two subcategories are,
respectively,
\[ \op{add}(E\oplus E_c^-) \quad\text{and}\quad \op{add}(E\oplus E_c^+). \]
\end{lemma}
\begin{proof}
Taking homology in \eqref{eq:e0r} gives $E_c^- \in \Gen(E)$, hence $\Gen(E)=\Gen(E\oplus E_c^-)$. The objects in $\op{add}(E\oplus E_c^-)$ are $\E$-projective, and the converse follows from Remark \ref{r:maxrigid}. The argument for $\check{\mc T}(\ep)$ is dual.
\end{proof}

\begin{remark} \label{r:Ec}
(1). The $\nu$-dual of Lemma \ref{L:Eproj} is as follows. 
The $\Ec$-injective representations in 
$\op{Cogen}(\tau E)$ and $\mc{F}(\ep)$ are $\op{add}(\tau E \oplus (\tau E)_c^-)$ and $\op{add}(\tau E \oplus (\tau E)_c^+)$ respectively.
	
(2). Unless $E$ is maximal rigid, one has $E_c^+ \notin \Gen(E)$. 
In particular, no summand of $e_c^+$ is negative.
Assume for contradiction that $E_c^+ \in \Gen(E)$. Then $\Gen(E \oplus E_c^+) = \Gen(E)$. 
But $\Gen(E \oplus E_c^+) = \check{\mc{T}}(\ep\oplus \ep_c^+) = \check{\mc{T}}(\ep)$ by Lemma \ref{L:Eproj}, and this implies $\Gen(E) = \check{\mc{T}}(\ep)$.
This only holds for $E$ maximal rigid.
\end{remark}

\section{The orthogonal projection functor} \label{S:OP}
Let $e$ be a rigid presentation with $E=\coker(e)$. For $M\in\rep(A)$, put $\br M=\eot e(M)$ and consider the torsion-pair sequence
\begin{equation}\label{eq:tf} 0\longrightarrow t_{\ep}(\br M)\longrightarrow\br M\longrightarrow f_{\ep}(\br M)\longrightarrow0. \end{equation}
Define
\[ L_e(M):=f_{\ep}(\br M), \qquad L_e=f_\ep\,\eot e. \]
Since $\E(e,\br M)=0$, one has $\tc_{\ep}(\br M)=\br M$ and hence $L_e(M)=\heo\ep(\br M)$. Replacing $\eot e$ by its reduced version $\eo e$ gives the same functor.

The canonical homomorphism $M\to L_e(M)$ is displayed in the diagram
	$$\xymatrix{& t_\ep(\br{M}) \ar@{_(-}[d] \\  M\ar[r] \ar[dr] & \br{M} \ar@{->>}[r]\ar@{->>}[d] & h_1E \\ & L_e(M)}$$
where the horizontal sequence is obtained from \eqref{eq:e0} by taking the homology.
	Hence there is a map $\Hom(M,N)\to\Hom(M,L_e(N))\cong\Hom(L_e(M),L_e(N))$, where the isomorphism follows from Theorem \ref{T:L_e}. This is the {\em right orthogonal projection} functor $L_e$.

\begin{remark}\label{r:Lneg} When $e=P_i[1]$, the functor $t_{\ep}$ is zero and $L_e(M)$ has the following description. Add minimal copies of $P_i[1]$ to $d_M$ to cover the kernel of $P_+\to M$ such that $\coker(d_M)$ is not supported on $i$. Then this new cokernel is $L_e(M)$.
\end{remark}

The existence of the adjoint follows from \cite[Proposition 3.3]{AHMV16ref}; the next theorem identifies it with the functor defined here.
\begin{theorem}\label{T:L_e} The functor $L_e=f_\ep\,\eot{e}: \rep(A)\to \rep(e^\perp)$ defined above is left adjoint to the inclusion functor $\iota_e: \rep(e^\perp) \to \rep(A)$.
\end{theorem}
\begin{proof} Let $N$ be any representation in $\rep(e^\perp)$. To verify the adjoint property 
\begin{equation}\label{eq:adj} \Hom(L_e(M),N)\cong \Hom(M, \iota_e(N)), \end{equation}
we apply $\Hom(-,N)$ to the triangle \eqref{eq:e0} for $d=d_M$, and get the exact sequence 
\[ 0=\Hom(h_1 e, N)\to \Hom(\br{d},N)\to \Hom(d_M,N)\to \E(h_1 e, N)=0 \]
with $\coker(\br{d})=\br{M}$.
So $\Hom(\br{M},N)\cong \Hom(d_M,N)$.
Since $t_{\ep}(\br{M}) \in \Gen(E)$, it follows that $\Hom(t_{\ep}(\br{M}),N)=0$. Then from the long exact sequence
	\[ 0 \to \Hom(f_{\ep}(\br{M}), N) \to \Hom(\br{M},N) \to \Hom(t_{\ep}(\br{M}), N)=0, \]
	we have $\Hom(f_{\ep}(\br{M}), N) \cong \Hom(\br{M},N)$.
	Finally,
\[\Hom(M,N)\cong\Hom(d_M,N)\cong\Hom(\br M,N)\cong\Hom(f_\ep(\br M),N).\]
\end{proof}
\noindent One also has $\E(M,N)=\E(d_M,N)\cong\E(\br d,N)$, whereas $\E(\br d,N)$ need not be isomorphic to $\E(f_\ep(\br M),N)$. Section \ref{ss:ereg} treats this discrepancy.

\begin{corollary}\label{C:equi} If $P_- \to P_+\to M\to 0$ is a projective presentation of $M$, then
	$L_e(P_-) \to L_e(P_+)\to L_{e}(M)\to 0$ is a projective presentation of $L_e(M)$.
In particular, $\rep(e^\perp)$ is equivalent to $\rep({_e}A)$ where ${_e}A = \End(\bigoplus \Ind(L_e(A)))$.
\end{corollary}
\begin{proof} A left adjoint to an exact functor is right exact and preserves projective presentations. 
In particular, $L_e(A)$ is a progenerator of $\rep(e^\perp)$, and $\rep(e^\perp) \cong \rep({_e}A)$.
\end{proof}
\noindent Since $L_e$ preserves projectives, it induces an additive functor on projective presentations, again denoted by $L_e$; in particular, $L_e(P_i[1])=L_e(P_i)[1]$.

The next lemma determines the action of $L_e$ on indecomposable projective representations.
Recall that $f_\ep(M)=0$ if and only if $M\in \Gen(E) = \mc{T}(\ep)$.
The following lemma is a variant of \cite[Lemmas 4.6 and 4.7]{BM1}; only the vanishing assumption in the proof changes.
\begin{lemma}[{{\em cf.} \cite[Lemma 4.6, 4.7]{BM1}}] \label{L:fep} 
	Let 
	\[ \mc{R}= \{ \text{Indecomposable } M\in \rep(A) \setminus \Gen(E) \mid \E(M,E)=0 \}. \]
	Then the functor $f_\ep$ is injective on $\mc{R}$, and preserves indecomposable modules.
\end{lemma}
\begin{proof} The proofs of \cite[Lemmas 4.6 and 4.7]{BM1} use the rigidity of $M\oplus E$ only to obtain $\Ext^1(M,t_\ep(N))=0$. Here this vanishing follows from
	\[ \Ext^1(M,t_\ep(N))\subseteq\E(M,t_\ep(N))\subseteq\E(M,hE)=0. \]
	The same arguments therefore prove preservation of indecomposability and injectivity on $\mc R$.
\end{proof}

\begin{corollary}\label{C:indP} The indecomposable projective representations in $\rep(e^\perp)$ are precisely given by $P_{i,e^\perp} = f_{\ep}(E_i^+)$.
\end{corollary}
\begin{proof} By Corollary \ref{C:equi} the projective representations in $\rep(e^\perp)$ are obtained from $L_e(A) = f_\ep(\eo{e}(A))$. Recall that $e_c^+ = \Ind(\eo{e}(A[0]))$.
By Remark \ref{r:Ec}.(2), $E_c^+ \notin \Gen(E)$.
By Lemma \ref{L:fep} indecomposable ones are precisely $P_{i,e^\perp}=f_\ep(E_i^+)$ and $P_{i,e^\perp} \ncong P_{j,e^\perp}$ for $i\neq j$. 
\end{proof}

\begin{remark} (1) For an injective presentation $\ec$, define the left orthogonal projection by 
\[ R_\ec = \tc_{-\epc}\, \eort{\ec}. \]	
Dually, $R_\ec$ is right adjoint to the inclusion $\rep({}^\perp\ec)\hookrightarrow\rep(A)$.
	
(2). If $e$ is rigid, then the injective presentation $\nu e$ is $\Ec$-rigid.
We have $\rep(e^\perp) \cong \rep({^\perp}\nu e)$ by the Auslander-duality.
Hence, $R_{\nu e}$ also projects to $\rep(e^\perp)$, and similarly $L_{\nu^{-1} \ec}$ projects to $\rep({^\perp \ec})$.
\end{remark}	

	
\begin{corollary}\label{C:bimod} There are natural isomorphisms 
\[ L_e(M) \cong \Hom(M, R_{\nu e}(A^*))^*\ \text{ and }\ R_{\ec}(M) \cong \Hom(L_{\nu^{-1} \ec}(A), M). \]
In particular, $(_eA)^* \cong \bigoplus \Ind (R_{\nu e}(A^*))$.
\end{corollary}
\begin{proof} By the adjoint property, we have that 
\[ L_e(M) \cong \Hom(L_e(M), A^*)^* = \Hom(L_e(M), R_{\nu e}(A^*))^* = \Hom(M, R_{\nu e}(A^*))^*. \]
The proof for $R_{\ec}(M)$ is similar.
\end{proof}

\begin{remark}
	(1). This recovers Jasso's \cite[Theorem 1.4]{Ja} in a different form.
	Jasso proved that $\rep({_e}A) \cong B/\innerprod{\underline{e}}$ where $B=\End(E^+)$ and $\underline{e}$ is the idempotent corresponding to $e$.
	Lemma \ref{L:factor} identifies the two descriptions.
	
	(2). Motivated by the work \cite{IT}, Buan-Marsh defined a category $\mf{W}_A$ associated to any basic algebra $A$ (see \cite[Theorem 0.2]{BM2}).
	Corollary \ref{C:bimod} identifies $\mf W_A$ with a subcategory of the standard category of algebras and bimodules.
\end{remark}

\noindent We denote by $L_{_eA}$ the functor $L_e$ followed by the equivalence $\rep(e^\perp) \cong \rep({_e}A)$.
Equivalently, $L_{_eA}(M) \cong \Hom(M, \bigoplus \Ind(R_{\nu e}(A^*)))^*$. We call it the right orthogonal projection to $_eA$.
Similarly, we have $R_{A_\ec}$ the left orthogonal projection to $A_\ec$.

\begin{corollary}\label{C:Le=fe}
	For $M\in\rep(e^{\rperp\,})$,
	\[ L_{_eA}(M)\cong\Hom(M,(\nu e)_c^+)^*. \]
	Dually, for $M\in\rep({^{\,\lperp}}\ec)$,
	\[ R_{A_{\ec}}(M)\cong\Hom((\nu^{-1}\ec)_c^+,M). \]
\end{corollary}
\begin{proof}
It is enough to identify $L_e(M)$ with
$\Hom(M,\eor{\nu e}(A^*))^*$. Since $M\in\rep(e^{\rperp\,})$,
\[ \Hom(M,\eor{\nu e}(A^*)) \cong\Hom(M,\eort{\nu e}(A^*)). \]
Consider the exact sequence
\[ 0\to R_{\nu e}(A^*)\to\eort{\nu e}(A^*) \to\fc_\ep(\eort{\nu e}(A^*))\to0. \]
The last Hom-space in
\[ 0\to\Hom(M,R_{\nu e}(A^*)) \to\Hom(M,\eort{\nu e}(A^*)) \to\Hom(M,\fc_\ep(\eort{\nu e}(A^*))) \]
vanishes because $\Hom(M,\nu e)=0$ and
$\fc_\ep(\eort{\nu e}(A^*))\in\op{Cogen}(\ker(\nu e))$.
Corollary \ref{C:bimod} now gives
\[ L_e(M)\cong\Hom(M,R_{\nu e}(A^*))^* \cong\Hom(M,\eort{\nu e}(A^*))^*. \]
\end{proof}

Set $E=\coker(e)$ and $C:=E_c^+=\coker(e_c^+)$ for the remainder of
this subsection.
\begin{lemma}\label{L:eck0}
Let $K$ be the kernel of the universal map $hE\to C$. Then
$\E(C,K)=0$.
\end{lemma}
\begin{proof}
Apply $\Hom(-,K)$ to the triangle
$h_1e[-1]\to A[0]\to\eot e(A[0])\to h_1e$. This gives
\[ \Hom(A,K)\to\E(h_1E,K)\to\E(\eot e(A),K)\to\E(A,K)=0. \]
Since $\E(C,K)\subseteq\E(\eot e(A),K)$, it remains to prove
$\E(E,K)=0$. Applying $\Hom(E,-)$ to
\begin{equation}\label{eq:KEC}
0\to K\to hE\to t_\ep(C)\to0
\end{equation}
gives
\[ \Hom(E,hE)\to\Hom(E,t_\ep(C))\to\E(E,K)\to\E(E,hE)=0. \]
Universality of $hE\to C$ makes the map
\[ \Hom(E,hE)\longrightarrow\Hom(E,C) \]
surjective. Hence the map to $\Hom(E,t_\ep(C))$ is also surjective, and
$\E(E,K)=0$.
\end{proof}

\begin{lemma}\label{L:factor}
Let $C=E_c^+=\coker(e_c^+)$. The functor $L_e$ induces a surjective algebra homomorphism
\[ \End(C)\longrightarrow \End(f_\ep(C)) \]
whose kernel is the ideal $[E](C,C)$ of endomorphisms of $C$ factoring through $\op{add}(E)$.
Consequently,
\[ \End(C)/[E](C,C)\cong \End(f_\ep(C)) \]
as algebras.
\end{lemma}

\begin{proof}
The adjunction in Theorem \ref{T:L_e} identifies
$\Hom(C,f_\ep(C))$ with $\End(f_\ep(C))$, and under this identification the map above is the map on endomorphism rings induced by the functor $L_e$.
From the torsion-pair sequence \eqref{eq:tf}, we obtain
\[ 0\to \Hom(C,t_\ep(C))\to \End(C) \longrightarrow \Hom(C,f_\ep(C))\to \E(C,t_\ep(C))=0, \]
where the last term vanishes by \eqref{eq:T}. Hence the map is surjective and its kernel is the image of $\Hom(C,t_\ep(C))$.

Let $K$ be as in Lemma \ref{L:eck0}. Applying $\Hom(C,-)$ to \eqref{eq:KEC} gives
\[ \Hom(C,K)\to \Hom(C,hE)\to \Hom(C,t_\ep(C)) \to \E(C,K)=0. \]
Thus every element of the kernel factors through $\op{add}(E)$.
Conversely, $\Hom(E,f_\ep(C))=0$, so every endomorphism of $C$ factoring through $\op{add}(E)$ is killed by $L_e$.
Therefore the kernel is exactly $[E](C,C)$.
\end{proof}

\section{Stable and Schur reductions} \label{S:SS}
\subsection{Simple objects and completion matrices}  \label{ss:simple}
In this subsection, we assume that $e$ is indecomposable. Recall from \cite{Ki} that a representation $M$ is called {\em $\delta$-stable} if $\delta\cdot \dv M=0$ and $\delta\cdot \dv S<0$ for any nontrivial subrepresentation $S$ of $M$.
\begin{definition}\label{D:Delta} For a maximal rigid presentation $\br{e} = \bigoplus_{i} e_i$, the {\em $\Delta$-matrix} of $\br{e}$ is the matrix $\Delta_{\br{e}}$ with rows given by $\ep_i$.
Its inverse is called the {\em $C$-matrix} of $\br{e}$ denoted by $\C_{\br{e}}$.
\end{definition}
\noindent Theorem \ref{T:maxrigid} implies that the matrix $\Delta_{\br{e}}$ is unimodular, so $\C_{\br{e}}$ is an integral matrix.

Write $\br E:=\coker(\br e)$.
\begin{theorem}[\cite{As,Tr}] \label{T:Cmat} Let $\br{e}=\bigoplus e_i$ be a maximal rigid presentation, and $e_{\hatj} = \bigoplus_{i\neq j} e_i$. 
Then there exists a unique $\ep_{\hatj}$-stable representation $S_j$,
whose dimension vector is given by the $j$-th column of $\pm \C_{\br{e}}$, 
where $+$ (resp. $-$) is picked if $S_j\in \Gen(\br{E})$ (resp. $S_j \in \op{Cogen}(\tau \br{E})$).
\end{theorem}
\noindent Denote the sign in the preceding theorem by $\sgnc(j,\br e)$, or by $\sgnc_j$ when $\br e$ is clear. The representation $S_j$ is the unique Schur representation in $\rep(e_{\hatj}^\perp)$ \cite{As,Tr}. The following refinement combines several results of \cite{As}.
\begin{theorem}[\cite{As}]\label{T:Schur_rad} The above $\ep_{\hatj}$-stable representation $S_j$ can be realized as 
$$\begin{cases}
	E_j / \sum_{f\in \op{rad}(\br{E}, E_j)} \Img(f) & \text{if $\sgnc_j=+$} \\
	\bigcap_{f\in \op{rad}(\tau E_j, \tau \br{E})} \Ker(f) & \text{if $\sgnc_j=-$.}
\end{cases}$$
\end{theorem}
\begin{proof} This follows from \cite[Definition 1.14, Proposition 1.17, and Theorems 2.3, 2.17]{As}.
\end{proof}

\begin{definition}\label{D:Schred}
The $\ep_{\hatj}$-stable representation $S_j$ in Theorem \ref{T:Cmat} is called the {\em $\ep_{\hatj}$-stable reduction} of $e=e_j$ and is denoted by $\rho_{\ep_{\hatj}-\st}(e)$. It is positive (resp. negative) when $\sgnc_j=+$ (resp. $\sgnc_j=-$).
When $e_{\hatj}=e_c^-$ and $\sgnc_j=+$, we denote $S_j$ by
$\rho_{+\st}(e)$ and call it the positive stable reduction of $e$.
When $e_{\hatj}=e_c^+$ and $\sgnc_j=-$, we denote $S_j$ by
$\rho_{-\st}(\tau e)$ and call it the negative stable reduction of
$\tau e$.
\end{definition}



The main cases are the completions $e^\pm=e\oplus e_c^\pm$. Place $\ep$ in the last row of $\Delta_{e^\pm}$ and denote the submatrix of the first $n-1$ rows by $\Delta_{e_c^\pm}$. Let $\C_{e^\perp}$ be the matrix formed by the first $n-1$ columns of $\C_{e^+}$, equivalently of $-\C_{e^-}$; see Lemma \ref{L:Cmat}(1).

\begin{lemma}\label{L:Cmat} We have the following interpretation for the columns of $\C_{e^\pm}$:  
\begin{enumerate}
\item 	The first $n-1$ columns of both $\C_{e^+}$ and $-\C_{e^-}$  are given by the dimension vectors of all simple representations in $\rep(e^\perp)$.
\item Whenever the corresponding stable reduction is defined,
$\C_{e^-}(-,n)=\dv\rho_{+\st}(e)$ and
$-\C_{e^+}(-,n)=\dv\rho_{-\st}(\tau e)$, respectively.
Otherwise, the corresponding column is $-\eu_i$ for
$e=P_i[1]$ or $e=P_i[0]$, respectively.
\end{enumerate}
\end{lemma}
\begin{proof} 
Recall from Corollary \ref{C:indP} that $P_{i,e^\perp}=f_\ep(e_i^+)$. Let $S_{j,e^\perp}$ be the simple representation in $\rep(e^\perp)$ covered by $P_{j,e^\perp}$. Then \begin{align} 
\label{eq:homecS} \hom(f_\ep(e_i^+), S_{j, e^\perp}) &= \hom(e_i^+, S_{j, e^\perp}) = {\delta}_{ij} \\
\notag \e(f_\ep(e_i^+), S_{j, e^\perp}) &= \e(e_i^+, S_{j, e^\perp}) = 0.
\end{align}
The first assertion follows from \eqref{eq:h-e} and Definition \ref{D:Delta}.
The description of the stable-reduction follows from Theorem \ref{T:Cmat} and Definition \ref{D:Schred}.
The negative completion of $P_i[1]$ is $A[1]$, while the positive completion
of $P_i[0]$ is $A[0]$; thus the relevant full $C$-matrices
are $-I$ and $I$, respectively.
\end{proof}

\begin{convention}[Ordering]\label{c:OC}
The rows of $\Delta_{e^\pm}$ have no preferred order. By Lemma \ref{L:Cmat}(1), they may be ordered so that the first $n-1$ columns of $\C_{e^+}$ and $-\C_{e^-}$ agree. This ordering is used throughout.
\end{convention}

The following criterion characterizes both signed complements.
\begin{lemma}[{\cite[Theorems 4.5 and 4.6]{CGY}}]\label{L:Bon}
Let $\ep$ be indecomposable, let $\bigoplus_i\delta_i$ be a complement of $\ep$, and let $\Delta_{\delta,\ep}$ be the matrix whose rows are the weights $\delta_i$ and $\ep$.
Then $\bigoplus_i\delta_i$ is the positive (resp. negative) complement of $\ep$ if and only if the columns of $\Delta_{\delta,\ep}^{-1}$ corresponding to the $\delta_i$ are all nonnegative (resp. nonpositive).
\end{lemma}

\begin{lemma}\label{L:Smulti} The matrix entry ${\C}_{e^\perp}(i,j) = \dim {S}_{j,e^\perp}(i)$ equals the multiplicity $m_{i,j}= [e_j^+ : \eo{e}(P_i)]$, which is also the multiplicity $[P_{j,e^\perp}: L_e(P_i)]$.
\end{lemma}
\begin{proof} Apply $\Hom(-, {S}_{j,e^\perp})$ to the triangle of presentations 
	\[ h_1e[-1]\to P_i \to \eot{e}(P_i) = a_ie\oplus \bigoplus_k m_{i,k}e_k^+ \to h_1e. \]
	We have that 
	\[ 0= \Hom(h_1e, S_{j,e^\perp}) \to \Hom(\bigoplus_k m_{i,k} e_k^+, S_{j,e^\perp})\to \Hom(P_i, S_{j,e^\perp})\to \E(h_1e, S_{j,e^\perp})=0. \]
	We find by \eqref{eq:homecS} that $\dim S_{j,e^\perp}(i) = m_{i,j}$.
Lemma \ref{L:fep}, equivalently Corollary \ref{C:indP}, shows that the modules $f_\ep(e_i^+)$ are pairwise nonisomorphic and indecomposable. Hence $m_{i,j}=[P_{j,e^\perp}:L_e(P_i)]$.
\end{proof}

\subsection{Schur reductions}
We give a module-theoretic description of the stable reductions, used again in Sections~\ref{S:bijection} and~\ref{S:inmu}. Theorem~\ref{T:Schred} also follows from \cite{As} via Theorem~\ref{T:Schur_rad}.

For a fixed indecomposable representation $X$, let $\Rcirc:=\Rcirc(X)$ be the radical of $\End(X)$.
By the Krull-Remak-Schmidt theorem (e.g., \cite{ASS}), $\Rcirc$ is just the nilpotent subspace of $\End(X)$.
We write $\Rcirc X=\sum_{f\in\Rcirc}\Img f$.

\begin{definition}\label{D:Schredbr}
For a nonzero indecomposable module $X$, define its positive and negative
Schur reductions by
\[ \rho_+(X)=X/\Rcirc X,\qquad \rho_-(X)=\bigcap_{f\in\Rcirc}\ker f. \]
Set $\rho_+(0)=\rho_-(0)=0$, and write
$\rho_\pm(d)=\rho_\pm(\coker d)$ whenever the cokernel is
indecomposable or zero. For any indecomposable presentation $d$, define
\[ \gamma_+(d)=\dv\rho_+(d),\qquad \gamma_-(d)=\dv\rho_-(\tau d), \]
except that
\[ \gamma_+(P_i[1])=-\eu_i,\qquad \gamma_-(P_i[0])=-\eu_i. \]
Thus the Schur reductions are defined at zero, but the signed vectors
at $P_i[1]$ and $P_i[0]$ are not their dimension vectors.
\end{definition}

The following lemma is a special case of \cite[Lemma 1.5]{As}.

\begin{lemma}\label{L:rhobr}
Let $X$ be an indecomposable module. Under the respective rigidity hypotheses:
\begin{enumerate}
\item If $\E(X,X)=0$, $\rho_+(X)$ is Schur and
$\hom(X,\rho_+(X))=1$, $\e(X,\rho_+(X))=0$.
\item If $\Ec(X,X)=0$, $\rho_-(X)$ is Schur and
$\hom(\rho_-(X),X)=1$, $\check\e(\rho_-(X),X)=0$.
\end{enumerate}
\end{lemma}
\begin{proof}
Set $\Rcirc=\Rcirc(X)$. Since $\Rcirc X\in\Gen(X)$ and
$\E(X,X)=0$, one has $\E(X,\Rcirc X)=0$. Applying $\Hom(X,-)$ to
$0\to\Rcirc X\to X\to\rho_+(X)\to0$ gives
\[ \Hom(X,\rho_+(X))\cong\End(X)/\Rcirc\cong K.\]
The quotient property gives $\E(X,\rho_+(X))=0$, while
$\End(\rho_+(X))$ embeds into $\Hom(X,\rho_+(X))$, proving (1).
Assertion (2) follows dually under $\check\E(X,X)=0$.
\end{proof}

\begin{theorem}\label{T:Schred}
Let $e$ be an indecomposable rigid presentation.
Whenever the corresponding stable reduction is defined, we have
\[\rho_+(e)\cong\rho_{+\st}(e),\qquad\rho_-(\tau e)\cong\rho_{-\st}(\tau e).\]
\end{theorem}
\begin{proof}
We prove the positive assertion; the negative one is dual.
Assume $\rho_{+\st}(e)$ is defined and set $E=\coker(e)$.
Put $\Rcirc=\Rcirc(E)$.
Every morphism $E_c^-\to E$ has image in $\Rcirc E$: otherwise composing
with an epimorphism $E^m\twoheadrightarrow E_c^-$ produces an invertible
endomorphism of $E$, forcing $E$ to split off $E_c^-$.
As $E_c^-$ is $\E$-projective in $\Gen(E)$, the sequence
$0\to\Rcirc E\to E\to \rho_+(E)\to0$ gives
\[ \Hom(E_c^-,\rho_+(E))=\E(E_c^-,\rho_+(E))=0.\]
Moreover, Lemma \ref{L:rhobr} yields
$\ep\cdot\dv\rho_+(E)=1$. Since the unique simple of
$\rep((e_c^-)^\perp)$ also has $\ep$-weight one, this identifies
$\rho_+(E)$ with that simple and hence with $\rho_{+\st}(e)$ by
Theorem \ref{T:Cmat}.
\end{proof}

\begin{corollary}\label{C:max}
Every positive stable reduction of $e$ is a quotient of $\rho_+(e)$,
and every negative stable reduction of $\tau e$ is a submodule of
$\rho_-(\tau e)$.
\end{corollary}
\begin{proof}
Suppose a positive stable reduction of $e$ exists. By
Theorem~\ref{T:Schur_rad}, it is a quotient of
$E=\coker(e)$ by a submodule containing $\Rcirc E$.
Hence it is a quotient of $\rho_+(e)$. The negative assertion
is dual.
\end{proof}

We summarize all the equivalent descriptions for the vectors $\gamma_+(e)$ and $\gamma_-(e)$ so far.
\begin{proposition}\label{P:Schred}
The vectors $\gamma_+(e)$ and $\gamma_-(e)$ admit the following descriptions.
For each sign, (1) and (3) apply whenever the corresponding stable
reduction is defined; (2) and (4) hold without restriction.
\begin{enumerate}
\item The dimension vectors of the unique simple objects in
${e_c^-}^\perp$ and ${e_c^+}^\perp$, respectively.
\item The last columns of $\C_{e^-}$ and $-\C_{e^+}$, respectively.
\item The coordinatewise maxima of the stable-reduction vectors:
\[\gamma_+(e)=\max_{\sgnc_j=+} \bigl\{\dv\rho_{\ep_{\hatj}-\st}(e)\bigr\},\qquad
\gamma_-(e)=\max_{\sgnc_j=-} \bigl\{\dv\rho_{\ep_{\hatj}-\st}(\tau e)\bigr\}.
\]
\item $\gamma_+(e)=\dv(\coker e)-\bigl([\eort{e}(P_i[1]):e]\bigr)_i$ and
$\gamma_-(e)=\dv(\tau e)-\bigl([\eot{e}(P_i):e]\bigr)_i$.
\end{enumerate}
\end{proposition}
\begin{proof}
Statements (1) and (2) follow from Theorem~\ref{T:Schred} and
Lemma~\ref{L:Cmat}. Statement (3) follows from Corollary~\ref{C:max},
since the Schur reduction is itself a stable reduction.
For (4), write
$\eort e(P_i[1])=n_i e\oplus U_i$, where
$U_i\in\operatorname{add}(e_c^-)$. The defining triangle gives
$\delta_{\eort e(P_i[1])}=\dim E(i)\,\epsilon-\eu_i$,
with $E=\coker(e)$. Pairing with the last column
$\gamma_+(e)$ of $\C_{e^-}$ gives
$n_i=\dim E(i)-\gamma_+(e)(i)$.
Dually, $\delta_{\eot e(P_i)}=\eu_i+\dim(\tau e)(i)\,\epsilon$;
pairing with $\gamma_-(e)=-\C_{e^+}(-,n)$ gives the second formula.
The same calculation applies in the projective and shifted-projective
cases, using the signed-vector conventions.
\end{proof}

\section{Extending $L_e$ to presentations} \label{S:ereg}
Throughout this section, $e$ is indecomposable, as in Subsection
\ref{ss:simple}; see Remark \ref{r:decomposable}. Thus
$e_c^\pm=\bigoplus_{i=1}^{n-1}e_i^\pm$, and $\C_{e^\perp}$ denotes the
$n\times(n-1)$ matrix fixed there.
\subsection{The extension $\L_e$} \label{ss:extOP}
Although $L_e$ induces an additive functor on presentations, it kills
$E$ and the negative-complement modules $E_i^-=\coker(e_i^-)$, and may
leave extra shifted projectives. To recover the lost information, 
we consider a modified version $\L_e$ of $L_e$.
It sends $e_i^-$ to $P_{i,e^\perp}[1]$ and takes
minimal presentations for other ordinary modules. The following lemma
describes the zero fibre; Theorem~\ref{T:bijection} gives the resulting
bijection on compatible presentations without $e$-summands.

\begin{lemma}\label{L:0fibre} $L_e^{-1}(0) \cap \rep({e}^{\rperp\,}) =  \Gen(E)$ and 
$L_e^{-1}(0) \cap \rep(\comp{e}) = \op{add}(E\oplus E_c^- )$.
\end{lemma}
\begin{proof} The inclusion $\Gen(E)\subseteq L_e^{-1}(0)\cap\rep({e}^{\rperp\,})$ follows from $f_\ep|_{\Gen(E)}=0$.
Conversely, if $L_e(M)=0$ and $\E(E,M)=0$, then $0=L_e(M) = f_{\ep}(M)$. This implies $t_{\ep}(M)=M$ and thus $M\in \Gen(E)$ by \eqref{eq:T}.
	
For the second equality, $E\oplus E_c^-\in L_e^{-1}(0)$ because
$E\oplus E_c^-\in\Gen(E)$.
The presentation $e\oplus e_c^-$ is compatible, so $E\oplus E_c^-\in\rep(\comp e)$. Hence $\op{add}(E\oplus E_c^-)\subseteq L_e^{-1}(0)\cap\rep(\comp e)$.
	
Conversely, if $M\in L_e^{-1}(0) \cap \rep(\comp{e})$, then $M\in \Gen(E) \cap \rep(\comp{e})$. 
{The presentation $e\oplus e_c^-$ is maximal rigid and has cokernel $E\oplus E_c^-$. Since $E$ is a direct summand of $E\oplus E_c^-$, the inclusion $M\in\Gen(E)$ implies $M\in\Gen(E\oplus E_c^-)$. The characterization recalled in Remark \ref{r:maxrigid} therefore gives $\E(E\oplus E_c^-,M)=0$.
On the other hand, compatibility gives $\E(M,E)=0$. Since $E_c^-\in\Gen(E)$ and the functor $\E(M,-)$ is right exact, it follows that $\E(M,E_c^-)=0$, and hence $\E(M,E\oplus E_c^-)=0$.
Remark \ref{r:maxrigid} now yields $M\in\op{add}(E\oplus E_c^-)$, as required.}
\end{proof}

\begin{definition}\label{D:Lext}
	Write a presentation, up to homotopy, as $d\simeq d_M\oplus P[1]$,
	where $d_M$ is minimal and $P$ is projective. For each nonzero
	$E_i^-:=\coker(e_i^-)$, let $m_i(M)$ be its multiplicity as a direct
	summand of $M$. Define the additive map on isomorphism classes
	\[ \L_e(d):=d_{L_e(M)}\oplus L_e(P)[1] \oplus\bigoplus_{i:E_i^-\ne0}P_{i,e^\perp}^{\,m_i(M)}[1], \]
	where $d_{L_e(M)}$ is the minimal presentation in $\rep(e^\perp)$.
	Thus ordinary negative-complement summands are recorded as negative
	projectives. If $e_i^-=P_j[1]$, Lemma \ref{L:Lext}(1) gives the same
	rule $\L_e(e_i^-)=P_{i,e^\perp}[1]$.
\end{definition}
Dually, write
$\dc\simeq\dc_M\oplus I[-1]$, with $\dc_M$ minimal and $I$ injective. 
For each nonzero $\check E_i^-:=\ker(\ec_i^-)$, let $\check m_i(M)$ be its
direct-summand multiplicity in $M$. Define
\[ \mc R_\ec(\dc):={}\dc_{R_\ec(M)}\oplus(0\to R_\ec(I)) \oplus \bigoplus_{i:\check E_i^-\ne0} (0\to I_{i,{}^\perp\ec}^{\,\check m_i(M)}). \]
In particular, $\mc R_\ec(\ec_i^-)=I_{i,{}^\perp\ec}[-1]$.

The next lemma relates the modified projection to the original module functor.
\begin{lemma}\label{L:Lext} \ \begin{enumerate}
\item If $P_i[1] = e_j^-$ is an indecomposable summand in $e_c^-$, then {$L_e(P_i) \cong P_{j,e^\perp}$.}
\item $\L_e(d)$ presents $L_e(\coker(d))$.
\end{enumerate}
\end{lemma}
\begin{proof}
{(1) Put $\C^-:=\C_{e^-}=\Delta_{e^-}^{-1}$. If $P_i[1]=e_j^-$, then the $j$-th row of $\Delta_{e^-}$ is $-\eu_i$. Hence, for $1\leq k\leq n-1$,
$\delta_{jk} =\Delta_{e^-}(j,-)\C_{e^-}(-,k) =-\eu_i\C_{e^-}(-,k) =-\C_{e^-}(i,k) =\C_{e^\perp}(i,k)$,
where the last equality follows from the definition of $\C_{e^\perp}$ and Convention \ref{c:OC}. By Lemma \ref{L:Cmat}(1),
$\C_{e^\perp}(i,k)=\dim S_{k,e^\perp}(i)$,
and Lemma \ref{L:Smulti} identifies this number with the multiplicity $[P_{k,e^\perp}:L_e(P_i)]$. Thus $P_{j,e^\perp}$ occurs with multiplicity one and every other indecomposable projective occurs with multiplicity zero. Therefore
$L_e(P_i)\cong P_{j,e^\perp}$.

(2) By Lemma \ref{L:decmnn}, we may write, up to homotopy equivalence,
\begin{equation}\label{eq:d} d \simeq d_M \oplus d_{M^-} \oplus P[1], \end{equation}
where $M^-$ is the direct sum of all indecomposable summands of $\coker(d)$ belonging to $\op{add}(E_c^-)$, the module $M$ has no such summand, and $P$ is projective. Then
\[ \coker(d)\cong M\oplus M^- \qquad\text{and}\qquad L_e(\coker(d))\cong L_e(M), \]
because $L_e(M^-)=0$. By Definition \ref{D:Lext},
\[ \L_e(d) \simeq d_{L_e(M)}\oplus P^-[1]\oplus L_e(P)[1], \]
where $P^-[1]:=\L_e(d_{M^-})$ is a shifted projective presentation in $\rep(e^\perp)$. Consequently,
\[ \coker(\L_e(d))\cong L_e(M)\cong L_e(\coker(d)), \]
so $\L_e(d)$ presents $L_e(\coker(d))$.}
\end{proof}
{\noindent By Lemma \ref{L:Lext}(2) and Theorem \ref{T:L_e},
\[ \Hom(d,\iota_e(N))\cong\Hom(\L_e(d),N) \]
for every $N\in\rep(e^\perp)$. Thus the modified projection satisfies the adjunction isomorphism \eqref{eq:adj} at the level of presented modules.}

\subsection{$e$-regular presentations} \label{ss:ereg}
The adjunction $L_e\dashv\iota_e$ extends to $\E$-spaces:
\[ \E(d,\iota_e(N))\cong\E(L_e(d),N) \qquad(N\in\rep(e^\perp)). \]
For the modified projection this may fail, since
$L_e(d)\simeq\L_e(d)\oplus Q_d[1]$ for a projective
$Q_d\in\rep(e^\perp)$. The additional summand contributes
$\Hom_{e^\perp}(Q_d,N)$ to the $\E$-space. This motivates
$e$-regularity.
\begin{definition}\label{D:ereg} A presentation $d$ is called {\em $e$-regular} if $\L_e(d)$ and $L_e(d)$ are homotopy equivalent.
	A $\delta$-vector is called $e$-regular if a general presentation of weight $\delta$ is $e$-regular.
A representation $M$ is called $e$-regular if $d_M$ is $e$-regular.
\end{definition}

\begin{lemma} \label{L:ereg} The following are equivalent \begin{enumerate}
	\item The presentation $d$ is $e$-regular.
	\item The $\delta$-vector of $\L_e(d)$ is given by $\delta \C_{e^\perp}$.
	\item $L_e(d_M) \simeq d_{L_e(M)}$ where $d_M$ is as in the decomposition \eqref{eq:d}.
	\item $\e(d, \iota_e(N)) = \e(\L_e(d), N)$ for any $N\in \rep(e^\perp)$.
	\item $\ker\bigl(\nu_{e^\perp}\L_e(d)\bigr)\cong R_{\nu e}\bigl(\ker(\nu d)\bigr)$.
\end{enumerate}
\end{lemma}

\begin{proof}
$(2)\Rightarrow(1)$: Lemma \ref{L:Lext}(2) and Corollary \ref{C:equi} show that the two presentations
$\L_e(d)$ and $L_e(d)$
both present $L_e(\coker(d))$. By Lemma \ref{L:decmnn}, each is, up to homotopy, the direct sum of the same minimal presentation of this module and a shifted projective. Their equal weights force the shifted projectives to have the same class in $K_0(\proj e^\perp)$ and therefore the same indecomposable multiplicities. Thus $\L_e(d)\simeq L_e(d)$.

$(1)\Rightarrow(4)$: By (1), $\E(\L_e(d),N)$ is the cokernel of $\Hom(L_e(P_+),N) \xrightarrow{} \Hom(L_e(P_-),N)$. By the adjoint property, this is also the cokernel of $\Hom(P_+,\iota_e(N)) \xrightarrow{} \Hom(P_-,\iota_e(N))$, which is $\E(d,\iota_e(N))$. 

$(4)\Rightarrow(2)$: The preceding argument gives $\E(d,\iota_e(N))\cong \E(L_e(d),N)$ for any $N\in \rep(e^\perp)$.
Moreover, Lemma \ref{L:Lext}(2) and Corollary \ref{C:equi} imply that $\L_e(d)$ and $L_e(d)$ have the same cokernel, so
$\Hom(\L_e(d),N)\cong\Hom(L_e(d),N)$.
Assumption (4), together with \eqref{eq:h-e}, therefore gives
$\delta_{\L_e(d)}\cdot\dv N =\delta_{L_e(d)}\cdot\dv N$
for every $N\in\rep(e^\perp)$. Taking $N$ to be each simple $S_{i,e^\perp}$ yields
$\delta_{\L_e(d)}=\delta_{L_e(d)}=\delta\C_{e^\perp}$,
which is (2).

$(1)\Leftrightarrow(3)$: We first note that every summand $e_i^-$ is $e$-regular. Indeed,
$\delta_{\L_e(e_i^-)}=-\eu_i=\ep_i^-\C_{e^\perp}$,
so this follows from $(2)\Rightarrow(1)$. Every remaining shifted projective $P_i[1]$ is $e$-regular by Definition \ref{D:Lext}. In the decomposition \eqref{eq:d}, the summands $d_{M^-}$ and $P[1]$ are therefore $e$-regular. By definition, $\L_e(d_M)=d_{L_e(M)}$.
Consequently, $d$ is $e$-regular if and only if $L_e(d_M)\simeq d_{L_e(M)}$, which is (3).

$(1)\Leftrightarrow(5)$: Write
$L_e(d)\simeq\L_e(d)\oplus Q[1]$, with $Q$ projective over
$B={}_eA$. Nakayama duality and the adjunction give
\[ R_{\nu e}\bigl(\ker(\nu d)\bigr) \cong \ker\bigl(\nu_B L_e(d)\bigr) \cong \ker\bigl(\nu_B\L_e(d)\bigr)\oplus\nu_B Q. \]
By Krull--Schmidt, (5) is equivalent to $Q=0$, i.e., to (1).
The comparison concern kernels, not the injective
presentations: the latter may have additional shifted-injective summands.
\end{proof}
\begin{corollary} \label{C:ereg} {Let $M$ be a representation with no direct summand in $\op{add}(E_c^-)$. Then $M$ is $e$-regular if and only if $\e(M,S_{i,e^\perp})=\ext^1(L_e(M),S_{i,e^\perp})$ for every $i$.}
\end{corollary}
\begin{proof}
{Because $d_{L_e(M)}$ is minimal and $S_{i,e^\perp}$ is simple,
\[ \E(d_{L_e(M)},S_{i,e^\perp}) \cong\Ext^1(L_e(M),S_{i,e^\perp}). \]
Together with adjunction and \eqref{eq:h-e}, this gives
\[ \hom(M,S_{i,e^\perp})-\ext^1(L_e(M),S_{i,e^\perp}) =\delta_{d_{L_e(M)}}(i). \]
On the other hand, \eqref{eq:h-e} gives
\[ \hom(M,S_{i,e^\perp})-\e(M,S_{i,e^\perp}) =\delta_M\cdot\dv S_{i,e^\perp} =(\delta_M\C_{e^\perp})(i). \]
Thus the equalities for all $i$ are equivalent to
$\delta_{d_{L_e(M)}}=\delta_M\C_{e^\perp}$.
Since $M$ has no summand in $\op{add}(E_c^-)$, Definition \ref{D:Lext} gives $\L_e(d_M)=d_{L_e(M)}$; the conclusion follows from Lemma \ref{L:ereg}(2).}
\end{proof}

\begin{corollary}\label{C:e^-reg} Each $e_i^\pm$ is $e$-regular. In particular, $\op{add}(e_c^-)\cap\op{add}(e_c^+)=\{0\}$.
\end{corollary}
\begin{proof}
	By Corollary \ref{C:indP}, $L_e(E_i^+)=P_{i,e^\perp}\neq 0$.
	For every $j$, $L_e(\coker e_j^-)=0$: if $e_j^-$ is shifted
projective, its cokernel is zero; otherwise this follows from Lemma
\ref{L:0fibre}. Hence no $e_i^+$ is isomorphic to an $e_j^-$, so
	$\op{Ind}(e_c^-)\cap\op{Ind}(e_c^+)=\varnothing$ and thus $\op{add}(e_c^-)\cap\op{add}(e_c^+)=\{0\}$.	
	Definition \ref{D:Lext} and Corollary \ref{C:indP} now give
	\[ \L_e(e_i^-)=P_{i,e^\perp}[1], \qquad \L_e(e_i^+)=P_{i,e^\perp}. \]
	Their weights are respectively $-\eu_i=\ep_i^-\C_{e^\perp}$ and $\eu_i=\ep_i^+\C_{e^\perp}$.
	Thus Lemma \ref{L:ereg}(2) shows that every $e_i^\pm$ is $e$-regular.
\end{proof}

\section{Presentation and weight bijections}\label{S:bijection}

\subsection{The presentation bijection and genericity}
\label{ss:presentation-bijection}

Throughout this subsection, $e$ is indecomposable, as in Section
\ref{S:ereg}. When $d\in\drep(e^\perp)$, its projective terms and
weight are computed in $\rep(e^\perp)$, with projective basis
$P_{i,e^\perp}$.

For $N\in\rep(e^\perp)$, distinguish its ambient minimal presentation
$d_N^A$ from its minimal presentation $d_N^{e^\perp}$ in the perpendicular
category. Put $U_N=\eort e(d_N^A)$, and let $U_N^\circ$ be obtained
by deleting all summands in $\op{add}(e\oplus e_c^-)$.
Define
\[ \mc I_e\!\left(d_N^{e^\perp}\oplus \bigoplus_i P_{i,e^\perp}^{\,m_i}[1]\right) :=U_N^\circ\oplus\bigoplus_i(e_i^-)^{m_i}. \]
The dual construction on injective presentations is denoted by
$\check{\mc I}_\ec$. These lifts are distinct from the reduced
truncations $\eor e$ and $\eo\ec$.

\begin{theorem}\label{T:bijection}
The maps
\[ \L_e:\drep(\comp e)\longrightarrow\drep(e^\perp) \qquad\text{and}\qquad \mc I_e:\drep(e^\perp)\longrightarrow\drep(\comp e) \]
are mutually inverse bijections on isomorphism classes.
\end{theorem}
\begin{proof}
Every compatible shifted projective $P_j[1]$ belongs to
$\op{add}(e\oplus e_c^-)$: compatibility gives $E(j)=0$, hence
$(E\oplus E_c^-)(j)=0$ because $E_c^-\in\Gen(E)$, and Remark
\ref{r:maxrigid} applies to the maximal rigid presentation $e^-$.
By Lemma \ref{L:0fibre}, each indecomposable object of
$\drep(\comp e)$ is therefore either an $e_i^-$ or a minimal
presentation $d_M$ with $M\notin\Gen(E)$.
The former map to the distinct negative projectives
$P_{i,e^\perp}[1]$. For the latter, $L_e(M)=f_\ep(M)$, and Lemma
\ref{L:fep} shows that these images are nonzero indecomposable modules,
distinct for distinct $M$. Thus $\L_e$ is injective on isomorphism
classes, by additivity and Krull--Schmidt.

Now let $N\in\rep(e^\perp)$. Universality gives $\E(U_N,e)=0$,
while $\E(e,N)=0$ and rigidity of $e$ give $\E(e,U_N)=0$.
Thus $U_N$ is compatible. Put $M_N=\coker(U_N)$.
Taking homology in the truncation triangle gives
\[ E^{h_N}\longrightarrow M_N\longrightarrow N\longrightarrow0, \qquad h_N=\e(d_N^A,e). \]
The kernel lies in $\Gen(E)$, and $N\in\mc F(\ep)$; hence the kernel
is $t_\ep(M_N)$ and $f_\ep(M_N)\cong N$.

The preceding classification shows that $U_N^\circ$ is the minimal
presentation of the module obtained from $M_N$ by removing all direct
summands in $\op{add}(E\oplus E_c^-)$. These summands are killed by
$L_e$, so $\L_e(U_N^\circ)\simeq d_N^{e^\perp}$.
Together with $\L_e(e_i^-)=P_{i,e^\perp}[1]$, this proves
$\L_e\mc I_e(d)\simeq d$ for every $d\in\drep(e^\perp)$.
Injectivity proves the other inverse identity and shows that
$\mc I_e$ is independent of all choices.
\end{proof}

The following comparison result is the homological input for the genericity theorem. For presentations $d,d'$, write
\[ \pi^e_{d,d'}:\E(d,d')\longrightarrow \E\bigl(L_e(d),L_e(d')\bigr) \]
for the comparison map induced by $L_e$ on projective presentations, where the space on the right is computed in $\rep(e^\perp)$.

\begin{lemma}\label{L:e0reg}
	Every presentation compatible with $e$ is $e$-regular. If $\e(d,e)=0$, then $\pi^e_{d,d'}$ is surjective for every presentation $d'$. It is an isomorphism whenever, in addition, $\e(e,d')=0$.
\end{lemma}

\begin{proof}Suppose first that $d$ is compatible with $e$, and put $B={}_eA$.
By additivity, we may assume that $d$ is indecomposable. The summands
of $e$ are sent to zero by both projections, and Corollary
\ref{C:e^-reg} treats the summands of $e_c^\pm$.
Every compatible shifted projective $P_j[1]$ belongs to
$\op{add}(e\oplus e_c^-)$: compatibility gives $E(j)=0$, hence
$(E\oplus E_c^-)(j)=0$ because $E_c^-\in\Gen(E)$, and Remark
\ref{r:maxrigid} applies to the maximal rigid presentation $e^-$.
By Lemma \ref{L:0fibre}, it remains to consider $d=d_M$ for an
indecomposable module $M\notin\Gen(E)$.

Here $N=L_e(M)=f_\ep(M)$ is nonzero and indecomposable by Lemma
\ref{L:fep}. Since $L_e(d_M)$ presents $N$, Lemma
\ref{L:decmnn} gives $L_e(d_M)\simeq d_N^B\oplus P[1]$,
where $d_N^B$ is the minimal projective presentation over $B$ and
$P$ is projective over $B$. If $N$ is projective, Corollary
\ref{C:indP} and the injectivity in Lemma \ref{L:fep} identify $M$
with the corresponding $E_i^+$. Corollary \ref{C:e^-reg} then gives
$P=0$.

Assume that $N$ is not projective. Then $M$ is not projective,
since $L_e$ preserves projectives. The Nakayama equivalence identifies
$\nu_A d_M$ with the minimal injective presentation of the
indecomposable module $\tau_A M$, and transports compatibility with
$e$ to compatibility with $\nu_A e$. The dual of Lemma \ref{L:fep}
therefore shows that $R_{\nu e}(\tau_A M)$ is either zero or
indecomposable: on this orthogonal class, $R_{\nu e}$ agrees with
the corresponding torsion subfunctor.

For every projective $A$-module $Q$, adjunction and Nakayama duality
give a natural isomorphism $R_{\nu e}(\nu_AQ)\cong\nu_B L_e(Q)$.
Indeed, both objects represent the functor
$Y\mapsto\Hom_A(Q,\iota_eY)^*$ on $\rep(B)$.
This isomorphism is natural on projective modules. Together with the
left exactness of the right adjoint $R_{\nu e}$, it gives
\[ R_{\nu e}(\tau_A M) \cong\ker\bigl(\nu_B L_e(d_M)\bigr) \cong\tau_B N\oplus\nu_B P. \]
Since $\tau_B N\neq0$, indecomposability forces $P=0$. Thus
$L_e(d_M)\simeq d_N^B=\L_e(d_M)$, proving $e$-regularity.

	Now assume $\e(d,e)=0$. Put $N=\coker(d')$, let
	$\bar d'=\eot{e}(d')$, and write $\bar N=\coker(\bar d')$. The
	definition $L_e=f_\ep\eot{e}$ gives a factorization
	\[ \E(d,d') \twoheadrightarrow \E(d,\bar d') \xrightarrow{\sim} \E(d,L_e(N)) \xrightarrow{\sim} \E\bigl(L_e(d),L_e(d')\bigr). \]
	Indeed, the first map is surjective by the right truncation triangle
	and $\E(d,e)=0$. The middle map is induced by
	\[ 0\longrightarrow t_\ep(\bar N)\longrightarrow\bar N\longrightarrow L_e(N)\longrightarrow0; \]
	it is an isomorphism because $t_\ep(\bar N)\in\Gen(E)$ and
	$\E(d,E)=0$. The last isomorphism follows from adjunction, since
	$L_e(d')$ presents $L_e(N)$. The composite is precisely
	$\pi^e_{d,d'}$, so it is surjective.
	
	If moreover $\e(e,d')=0$, then $\bar d'=d'$, and the first map is also
	an isomorphism. Hence $\pi^e_{d,d'}$ is an isomorphism.
\end{proof}

\begin{corollary}\label{C:isoE} Assume $\e(d,e)=0$ and choose
$L_e(d)\simeq\L_e(d)\oplus Q_d[1]$. For every $d'$ there is a
surjection
\[ \E_A(d,d')\twoheadrightarrow \E_{e^\perp}\bigl(\L_e(d),\L_e(d')\bigr).\]
If $\e(e,d')=0$, its kernel is
$\Hom_{e^\perp}\bigl(Q_d,L_e(\coker d')\bigr)$.
Thus it is an isomorphism precisely when this Hom-space vanishes.
\end{corollary}
\begin{proof}
Compose $\pi^e_{d,d'}$ in Lemma \ref{L:e0reg} with
projection onto the first summand. Under the second vanishing
hypothesis $\pi^e_{d,d'}$ is an isomorphism, and the discarded
summand is
$\E_{e^\perp}(Q_d[1],L_e(d'))
=\Hom_{e^\perp}(Q_d,L_e(\coker d'))$.
\end{proof}

\begin{proposition}\label{P:induced-general}
	Let $\delta$ be a weight satisfying $\e(\delta,e)=0$. Write
	\[ \PHom(\delta)=\Hom_A(P_-,P_+),\qquad Q_\pm=L_e(P_\pm), \]
	and set
	\[ U_\delta=\{d\in\PHom(\delta)\mid\E(d,e)=0\},\quad W_\delta=\Hom_{e^\perp}(Q_-,Q_+), \quad H_\delta=\Aut(Q_-)\times\Aut(Q_+). \]
	Then $H_\delta\cdot L_e(U_\delta)$ is a nonempty open subset of $W_\delta$. Consequently, $L_e$ sends a general presentation of weight $\delta$ to a presentation homotopy equivalent to a general presentation of weight $\delta\C_{e^\perp}$. Moreover, $L_e$ sends every rigid presentation $d$ satisfying $\E(d,e)=0$ to a rigid presentation.
\end{proposition}

\begin{proof}
	The assumption $\e(\delta,e)=0$ makes $U_\delta$ a dense open subset of
	$\PHom(\delta)$. Moreover, Lemma \ref{L:Smulti} gives
	\[ [Q_+]-[Q_-]=\delta\C_{e^\perp}. \]
	
	Let
	\[ F:U_\delta\longrightarrow W_\delta,\qquad d\longmapsto L_e(d), \]
	and consider
	\[ \Phi:H_\delta\times U_\delta\longrightarrow W_\delta,\qquad (h,d)\longmapsto h\cdot F(d). \]
	For $d\in U_\delta$, the map induced by $dF_d$ on the normal spaces to the
	base-change orbits is
	\[ \pi^e_{d,d}:\E(d,d)\longrightarrow\E(F(d),F(d)), \]
	which is surjective by Lemma \ref{L:e0reg}. Hence
	\[ T_{F(d)}W_\delta=T_{F(d)}(H_\delta\cdot F(d))+dF_d(T_dU_\delta), \]
	so $d\Phi_{(1,d)}$ is surjective. By equivariance this holds at every point
	of $H_\delta\times U_\delta$. Both $H_\delta\times U_\delta$ and $W_\delta$ are smooth. Thus $\Phi$ is smooth and open, and therefore
	$O:=H_\delta\cdot F(U_\delta)$
	is a nonempty open subset of $W_\delta$.
	
	Let $V_{\mathrm{gen}}\subseteq W_\delta$ be a nonempty
	$H_\delta$-invariant open subset whose elements are homotopy equivalent to
	general presentations of weight $\delta\C_{e^\perp}$; such a set exists by
	the usual cancellation statement. Since $W_\delta$ is irreducible,
	$O\cap V_{\mathrm{gen}}\neq\varnothing$. The $H_\delta$-invariance of
	$V_{\mathrm{gen}}$ then implies
	\[ F^{-1}(V_{\mathrm{gen}})\neq\varnothing. \]
	This inverse image is open in the irreducible space $\PHom(\delta)$, hence
	dense. Therefore $L_e$ sends a general presentation of weight $\delta$ to a
	presentation homotopy equivalent to a general presentation of weight
	$\delta\C_{e^\perp}$.
	
	Finally, if $d$ is rigid and $\E(d,e)=0$, Lemma \ref{L:e0reg} gives a
	surjection
	\[ \E(d,d)\twoheadrightarrow\E(L_e(d),L_e(d)). \]
	The source vanishes, so $L_e(d)$ is rigid.
\end{proof}

\begin{theorem}\label{T:bijection_space}
    Let $\delta$ satisfy $\e(\delta,e)=0$. Then $\L_e$ sends a general
    presentation of weight $\delta$ to a general presentation in
    $\drep(e^\perp)$. If $\delta$ is $e$-regular, the image has weight
    $\delta\C_{e^\perp}$.

    Moreover, $\L_e$ sends every rigid presentation $d$ satisfying
    $\E(d,e)=0$ to a rigid presentation, and $\mc I_e$ sends every rigid
    presentation in $\drep(e^\perp)$ to a rigid presentation in
    $\drep(\comp e)$.
\end{theorem}
\begin{proof}
    For every presentation $d$, there is a projective
    $Q_d\in\rep(e^\perp)$ such that
    \[ L_e(d)\simeq\L_e(d)\oplus Q_d[1]. \]
    If $d$ is general of weight $\delta$, Proposition
    \ref{P:induced-general} shows that $L_e(d)$ is general of weight
    $\delta\C_{e^\perp}$. Hence canonical decomposition implies that
    $\L_e(d)$ is general in its weight. If $\delta$ is $e$-regular, then
    $Q_d=0$, so this weight is $\delta\C_{e^\perp}$.

    If $d$ is rigid and $\E(d,e)=0$, then $L_e(d)$ is rigid by
    Proposition \ref{P:induced-general}; hence its direct summand
    $\L_e(d)$ is rigid. Conversely, let $\bar d\in\drep(e^\perp)$ be
    rigid and put $y=\mc I_e(\bar d)$. Theorem \ref{T:bijection} gives
    $y\in\drep(\comp e)$ and $\L_e(y)\simeq\bar d$, while Corollary
    \ref{C:isoE} gives
    \[ \E(y,y)\cong\E(\bar d,\bar d)=0. \]
    Thus $y$ is rigid.
\end{proof}

Consequently, the bijection
$\L_e:\drep(\comp e)\xrightarrow{\sim}\drep(e^\perp)$
restricts to a bijection between the rigid presentations on the two sides.
This recovers the bijection on rigid presentations from \cite{BM1,BM2}.

\begin{remark}
	In \cite{BM3}, the authors defined the local mutation operation on
	$\tau$-exceptional sequences
	\[ (d,e)\longmapsto\bigl(e'=\mc I_d(e),\,\L_{e'}(d)\bigr). \]
	The following diagram explains the resulting symmetric-group action:
	\[
	\xymatrix{
		[d,e] \ar[d] \ar[r] &[e,d] \ar[d]\\
		(d,\L_d(e)) \ar[r] &(e,\L_e(d)).
	}
	\]
The upper row is the interchange of two compatible indecomposable rigid presentations.
The lower row is the mutation of the $\tau$-exceptional pair given by Theorem \ref{T:bijection_space}.
\end{remark}

\begin{corollary}[Factorization of the projection]\label{C:factor}
	Suppose that $e=e_1\oplus e_2$ is rigid, with $e_1$ indecomposable, and let
	$e_2'=\L_{e_1}(e_2)$.
	Then $e_2'$ is rigid and there is a natural isomorphism of module functors
	\[ L_e\cong L_{e_2'}L_{e_1}:\rep(A)\longrightarrow\rep(e^\perp), \]
	where $L_{e_2'}$ is formed inside $\rep(e_1^\perp)$.
\end{corollary}
\begin{proof}
	The presentation $e_2$ is compatible with $e_1$, so Theorem \ref{T:bijection_space} shows that $e_2'$ is rigid. For every $N\in\rep(e_1^\perp)$, adjunction and Corollary \ref{C:isoE} give
	\[ \Hom(e_2',N)\cong\Hom(e_2,N), \qquad \E(e_2',N)\cong\E(e_2,N). \]
	Hence the right perpendicular category of $e_2'$ inside $\rep(e_1^\perp)$ is precisely $\rep(e^\perp)$. The functor $L_e$ is left adjoint to the inclusion of this category into $\rep(A)$, while $L_{e_2'}L_{e_1}$ is another left adjoint to the same inclusion. The assertion follows from uniqueness of left adjoints.
\end{proof}

\begin{remark}[Arbitrary basic rigid presentations]
	\label{r:decomposable}\label{P:Cred-general}
	The results below extend to an arbitrary rigid presentation after replacing it by its basic version. Write $e=\bigoplus_{a=1}^r e_a$ and $e_c^\pm=\bigoplus_{i=1}^{n-r}e_i^\pm$. 
	In the truncation triangles, the multiples of $e$ are replaced by the corresponding universal $\op{add}(e)$-approximations
	\[ \bigoplus_{a=1}^r \e(e_a,d)e_a[-1]\longrightarrow d, \qquad d\longrightarrow\bigoplus_{a=1}^r \e(d,e_a)e_a[1], \]
	and the reduced truncations are obtained by removing all summands in $\op{add}(e)$.
	
	The category $\rep(e^\perp)$ has $n-r$ simple objects. Order $ e^\pm=e_c^\pm\oplus e $
	with the complementary summands first. Then the first $n-r$ columns of $\varepsilon\C_{e^\varepsilon}$ are the dimension vectors of the simple objects of $\rep(e^\perp)$, ordered dually to the projectives $L_e(e_i^+)$. In particular, these columns are independent of $\varepsilon$; their common matrix is denoted by
	\[ \C_{e^\perp}\in M_{n\times(n-r)}(\mb Z_{\geq0}). \]
	
	Define the modified projection as in Definition \ref{D:Lext}, with
	\[ \L_e(e_i^-)=P_{i,e^\perp}[1], \qquad 1\leq i\leq n-r. \]
	Define $\mc I_e$ by the same construction, deleting all summands in $\op{add}(e\oplus e_c^-)$ from the ordinary-module lift.
	With these conventions, Lemmas \ref{L:0fibre}, \ref{L:Lext}, and \ref{L:ereg}, Corollaries \ref{C:ereg} and \ref{C:e^-reg}, Theorem \ref{T:bijection}, Lemma \ref{L:e0reg}, Corollary \ref{C:isoE}, Proposition \ref{P:induced-general}, and Theorem \ref{T:bijection_space} remain valid, with the same proofs after replacing $n-1$ by $n-r$. The corresponding dual statements for $\mc R_{\ec}$ remain valid as well. Corollary \ref{C:factor} also shows that the ordinary module projection $L_e$ may be computed successively along any ordered decomposition of $e$ into indecomposable summands.
\end{remark}

\subsection{The weight bijection and canonical decompositions}
\label{ss:weight-bijection}

Let $e=\bigoplus_{a=1}^r e_a$
be a basic rigid presentation, and set
\[
\mc W_e
:=
\left\{
\delta\in\mb Z^{Q_0}
\ \middle|\
\text{a general presentation of weight $\delta$ belongs to
	$\drep(\comp{e})$}
\right\}.
\]
Consider the map
\[ p_e:\mc W_e\longrightarrow\mb Z^{n-r}, \qquad p_e(\delta)=\delta\C_{e^\perp}. \]

\begin{theorem}[Canonical decompositions under projection]
	\label{T:CDprojection}
	The map $p_e$ is a bijection. Denote its inverse by
	$\ell_e:\mb Z^{n-r}\to\mc W_e$.
	For any $\delta,\eta\in\mc W_e$, one has
	\[ \e_A(\delta,\eta) = \e_{e^\perp}\bigl(p_e(\delta),p_e(\eta)\bigr). \]
	
	Moreover, if
	\[ \delta=\delta_1\oplus\cdots\oplus\delta_s \]
	is the canonical decomposition of $\delta\in\mc W_e$, then every
	$\delta_i$ belongs to $\mc W_e$, and
	\[ p_e(\delta) = p_e(\delta_1)\oplus\cdots\oplus p_e(\delta_s) \]
	is the canonical decomposition of $p_e(\delta)$ in $\rep(e^\perp)$.
	Conversely, if
	\[ \bar\delta = \bar\delta_1\oplus\cdots\oplus\bar\delta_s \]
	is a canonical decomposition in $\rep(e^\perp)$, then
	\[ \ell_e(\bar\delta) = \ell_e(\bar\delta_1)\oplus\cdots\oplus \ell_e(\bar\delta_s) \]
	is its canonical decomposition. In particular, $p_e$ and $\ell_e$
	induce mutually inverse bijections between the indecomposable weights
	in $\mc W_e$ and the indecomposable weights of $\rep(e^\perp)$.
\end{theorem}

\begin{proof}
	By Remark \ref{r:decomposable}, the preceding results apply to $e$. Write $\ep_a$ for the weight of $e_a$. If
	$e^+=e_c^+\oplus e$ is a completion, with the complementary
	summands first, then unimodularity of $\Delta_{e^+}$ gives
	\[ \ker\bigl(\gamma\longmapsto\gamma\C_{e^\perp}\bigr) = \bigoplus_{a=1}^r\mb Z\ep_a. \]
	
	If a weight $\theta$ admits a
	presentation compatible with $e$, then every canonical summand of
	$\theta$ is compatible with $e$, by additivity and nonnegativity of
	generic $\E$-values. Consequently, after deleting all $\ep_a$-summands
	from its canonical decomposition, the remaining weight belongs to
	$\mc W_e$. In particular, every canonical summand of a weight in
	$\mc W_e$ again lies in $\mc W_e$, and none equals an $\ep_a$.
	
	Let $\bar\delta\in\mb Z^{n-r}$, choose a presentation $\bar d$ of
	weight $\bar\delta$ in $\drep(e^\perp)$, and let
	$y=\mc I_e(\bar d)$ have weight $\alpha$. By Theorem
	\ref{T:bijection}, $y$ is compatible with $e$ and
	$\L_e(y)\simeq\bar d$. Hence Lemmas \ref{L:e0reg} and
	\ref{L:ereg}(2) give
	\[ \alpha\C_{e^\perp}=\bar\delta. \]
	Delete the $\ep_a$-summands from the canonical decomposition of
	$\alpha$, and denote the remaining weight by $\beta$. The preceding observation gives $\beta\in\mc W_e$, while
	$\ep_a\C_{e^\perp}=0$ gives
	\[ p_e(\beta)=\bar\delta. \]
	Thus $p_e$ is surjective.
	
	If $p_e(\delta)=p_e(\eta)$ for $\delta,\eta\in\mc W_e$, then
	$\delta-\eta=\sum_a c_a\ep_a$. Writing
	$c_a^\pm=\max(\pm c_a,0)$, we obtain
	\[ \delta\oplus\bigoplus_a c_a^-\ep_a = \eta\oplus\bigoplus_a c_a^+\ep_a. \]
	Both sides are canonical by the preceding observation, rigidity of $e$,
	and Theorem \ref{T:CDPHom}. Since no $\ep_a$ occurs in the canonical
	decomposition of a weight in $\mc W_e$, uniqueness gives
	$c_a^-=c_a^+=0$ for every $a$. Hence $\delta=\eta$, and $p_e$ is
	bijective.
	
	Now let $\delta,\eta\in\mc W_e$. For
	$\theta\in\{\delta,\eta\}$, restrict the open map used in the proof of
	Proposition \ref{P:induced-general} to the dense open compatible locus
	in $\PHom(\theta)$. The product of the resulting open images meets the
	invariant dense open locus on which the projected pair realizes its
	generic $\E$-value. Pulling back and intersecting with the dense open
	locus on which the original pair realizes its generic $\E$-value, we
	may choose compatible presentations $d,d'$ such that both
	$(d,d')$ and
	$\bigl(\L_e(d),\L_e(d')\bigr)$ realize the corresponding generic
	$\E$-values. Corollary \ref{C:isoE} therefore gives
	\[ \e_A(\delta,\eta) = \e_{e^\perp}\bigl(p_e(\delta),p_e(\eta)\bigr). \]
	
	Finally, let
	$\delta=\delta_1\oplus\cdots\oplus\delta_s$
	be canonical with $\delta\in\mc W_e$. The preceding observation gives
	$\delta_i\in\mc W_e$. The forward part of Theorem
	\ref{T:bijection_space}, together with Theorem \ref{T:bijection},
	shows that each $p_e(\delta_i)$ is indecomposable. Moreover,
	\[ \e_{e^\perp}\bigl(p_e(\delta_i),p_e(\delta_j)\bigr) = \e_A(\delta_i,\delta_j) = 0 \qquad(i\ne j). \]
	Hence Theorem \ref{T:CDPHom} gives
	\[	p_e(\delta)	=	p_e(\delta_1)\oplus\cdots\oplus p_e(\delta_s).	\]
	
	Conversely, write the canonical decomposition of
	$\ell_e(\bar\delta)$ as
	\[ \ell_e(\bar\delta)=\eta_1\oplus\cdots\oplus\eta_t. \]
	The forward assertion gives a canonical decomposition
	\[ \bar\delta = p_e(\eta_1)\oplus\cdots\oplus p_e(\eta_t). \]
	Comparison with
	$\bar\delta=\bar\delta_1\oplus\cdots\oplus\bar\delta_s$,
	followed by uniqueness and injectivity of $p_e$, yields, after
	reordering, $\eta_i=\ell_e(\bar\delta_i)$.
	This proves the converse and the assertion on indecomposable weights.
\end{proof}

\begin{proposition}\label{P:delta}
Let $e$ be indecomposable rigid and $\bar\delta\in\mb Z^{n-1}$.
Put $\theta_a=-\bar\delta\Delta_{e_c^-}+a\epsilon$ for $a\in\mb Z$.
There is a unique $a_0\in\mb Z_{\ge0}$ such that
\[ \e(\theta_a,\epsilon)=\e(\epsilon,\theta_a)=0 \quad\Longleftrightarrow\quad a\ge a_0. \]
The inverse weight is $\alpha:=\ell_e(\bar\delta)=\theta_{a_0}$.
For every $a\ge a_0$,
$\theta_a=\alpha\oplus\bigoplus_{j=1}^{a-a_0}\epsilon$;
that is, its canonical decomposition consists of the 
summands of $\alpha$ and exactly $a-a_0$ copies of $\epsilon$.
For a general presentation $\bar d$ of weight $\bar\delta$ in
$\rep(e^\perp)$, the lift $\mc I_e(\bar d)$ has weight $\alpha$.
\end{proposition}
\begin{proof}
By Theorem \ref{T:CDprojection}, the fibre of $\bar\delta$
contains a unique weight $\alpha\in\mc W_e$, and it satisfies
$\alpha=\theta_{a_*}$ for some integer $a_*$. It is compatible with
$e$, so both generic $\E$-values vanish.
If $E=\coker e\neq0$, pairing with
$\gamma_+(e)=\dv\rho_+(e)$ gives
\[ a_*=\alpha\gamma_+(e)^\top =\hom(\alpha,\rho_+(e))\ge0:\]
the $\E$-term vanishes since $\rho_+(e)$ is a quotient of $E$.
If $E=0$, then $e=P_i[1]$. A compatible $e$-free presentation
has neither a positive $P_i$-term nor a negative $P_i[1]$-summand,
so $a_*=-\alpha(i)\ge0$.

If both generic $\E$-values vanish for some $\theta_a$,
its generic canonical decomposition is compatible with $e$.
Deleting its $e$-summands leaves a weight $\beta\in\mc W_e$
with $\beta\C_{e^\perp}=\bar\delta$. Uniqueness gives
$\beta=\alpha$, hence $\theta_a=\alpha\oplus b\epsilon$
for some $b\ge0$ and $a=a_*+b$. Conversely, for every $b\ge0$,
compatibility of $\alpha$ with $\epsilon$ and rigidity of $e$ give
$\alpha+b\epsilon=\alpha\oplus\bigoplus_{j=1}^{b}\epsilon$ by
Theorem \ref{T:CDPHom}; both generic $\E$-values therefore vanish.
Since no $\epsilon$ occurs among the canonical summands of
$\alpha\in\mc W_e$, its multiplicity is exactly $b$.
Taking $a_0=a_*$ proves the threshold and multiplicity assertions.
Finally, the open-image argument of Proposition
\ref{P:induced-general}, applied to $\alpha$, shows that the
inverse lift of a general reduced presentation has weight $\alpha$.
\end{proof}

\section{Quivers with potentials}  \label{S:QP}
\subsection{Jacobian realizations}  \label{ss:QP}
We use the terminology of \cite{DWZ1}. Let $(Q,\S)$ be a QP with
$Q$ finite and without loops or $2$-cycles. Write $B(Q)$ for its
exchange matrix and
\[ J=J(Q,\S)=\widehat{KQ}/\partial\S,\qquad \rep(Q,\S)=\rep(J).\]
All Jacobian algebras considered below are finite-dimensional.
Mutated QPs are taken in the reduced DWZ sense; the auxiliary
complete QPs of Proposition~\ref{P:KLend} may have loops or $2$-cycles
and are used only through their Jacobian algebras. All rigid
presentations and completions are basic.

Mutation $\mu_k$ acts on QPs and decorated representations.
Throughout the remainder of the paper, we do not assume nondegeneracy on QPs,
but each mutation sequence used is assumed \emph{admissible}: 
its mutations are defined and every intermediate reduced QP is $2$-acyclic. 
This applies also to complementary exchanges and conjugated sequences. 
A presentation is \emph{reachable} if an ordinary mutation sequence takes it to a negative projective, 
and \emph{extended-reachable} if the global operations $\tau^{\pm}$ are
also allowed. These operations commute with mutation and leave the
QP unchanged \cite[Proposition 7.10]{DF}. For an extended-reachable
cluster $d$, choose a sequence $\mub$ with $\mub(d)=J[1]$ and set
\[ (Q_d,\S_d)=\mub(Q,\S).\]

Let $\CQ$ be the Amiot's cluster category of $(Q,\S)$ and let $\b T$ be its canonical cluster-tilting object \cite{Am}.
The functor
\[ F_{Q,\S}=\Hom_{\CQ}(\b T,-) \]
induces an equivalence
\[ \CQ/[\Sigma\b T]\simeq\rep(Q,\S). \]
Plamondon's extension
\[ \wtd F_{Q,\S}:\mc D\longrightarrow\drep(Q,\S) \]
incorporates decorated representations and presentations \cite{P0,P}.
Thus rigid presentations correspond to rigid objects of $\CQ$, and clusters correspond to cluster-tilting objects.
We write $\b d$ for the object corresponding to a rigid presentation $d$.

\begin{proposition}[Keller--Liu]\label{P:KLend}
Let $X=\bigoplus_{i=1}^mX_i$ be a basic cluster-tilting object of $\CQ$.
There exists a Jacobi-finite complete QP $(Q_X,\S_X)$, possibly with loops or $2$-cycles, and an algebra isomorphism
\[ J(Q_X,\S_X)\cong\End_{\CQ}(X) \]
that sends the vertex idempotent labelled by $X_i$ to $1_{X_i}$.
No reachability hypothesis on $X$ is required.
\end{proposition}

\begin{proof}
Apply \cite[Theorem 6.2.1]{KL23ref} to the standard enhanced $2$-Calabi--Yau category $\CQ$ and the prescribed cluster-tilting object $X$.
It produces a $3$-dimensional deformed dg preprojective algebra $\Pi_X$ whose cluster category is equivalent to $\CQ$ and whose canonical cluster-tilting object corresponds to $X$.
By \cite[Remark 6.2.2]{KL23ref},
\[ H^0(\Pi_X)\cong\End_{\CQ}(X). \]
Since $K$ is algebraically closed and $X$ is basic, the semisimple base algebra is $K^m$.
In dimension $3$, the construction in \cite[Section 4.5]{KL23ref} is, after choosing dual bases, the complete Ginzburg dg algebra of a complete QP $(Q_X,\S_X)$.
Thus $H^0(\Pi_X)=J(Q_X,\S_X)$; ordering the primitive idempotents by the summands $X_i$ gives the stated compatibility.
\end{proof}

\begin{proposition}[Keller--Yang; Amiot]\label{P:QPend}
Let $d=\bigoplus_i d_i$ be an extended-reachable cluster and let $(Q_d,\S_d)$ be defined by an associated sequence.
Then there is an idempotent-compatible algebra isomorphism
\[ J(Q_d,\S_d)\cong\End_{\CQ}(\b d), \]
where the idempotent corresponding to $d_i$ maps to $1_{\b d_i}$.
\end{proposition}

\begin{proof}
Along the ordinary mutation steps, the Keller--Yang equivalences identify the canonical cluster-tilting object of the terminal QP with the mutated cluster-tilting object and preserve the labels of its indecomposable summands \cite{KY}.
Amiot's realization identifies the endomorphism algebra of the canonical object with the terminal Jacobian algebra \cite{Am}.
The global operations are simultaneous suspensions, which leave both the QP and the endomorphism algebra unchanged.
\end{proof}

\subsection{Restriction and Calabi--Yau reduction}\label{ss:OPinQP}
Let $I\subseteq Q_0$ and let $Q|_I$ be the full subquiver on $I$.
The continuous homomorphism
$\psi_I:\widehat{KQ}\longrightarrow\widehat{K(Q|_I)}$
sends an arrow to itself if both endpoints lie in $I$ and to $0$ otherwise.

\begin{definition}
The \emph{restriction} of $(Q,\S)$ to $I$ is
$(Q,\S)|_I=(Q|_I,\psi_I(\S))$. 
\end{definition}

\begin{lemma}\label{L:epi}
For any complete QP, $\psi_I$ induces an epimorphism
\[ J(Q,\S)\twoheadrightarrow J((Q,\S)|_I). \]
If $p=\sum_{i\notin I}p_i$, then
\[ J((Q,\S)|_I)\cong J(Q,\S)/J(Q,\S)pJ(Q,\S). \]
For a QP in the usual DWZ setting, restriction induces an epimorphism
\[ \operatorname{Def}(Q,\S)\twoheadrightarrow\operatorname{Def}((Q,\S)|_I) \]
and preserves rigidity and nondegeneracy.
\end{lemma}

\begin{proof}
For every arrow $a$ of $Q|_I$, one has
$\psi_I(\partial_a\S)=\partial_a\psi_I(\S)$, while the cyclic derivative with respect to a deleted arrow maps to zero.
Hence the image of the Jacobian ideal is the Jacobian ideal of the restriction.
Killing all paths through deleted vertices is precisely quotienting by the ideal generated by $p$.
The deformation-space statement is \cite[Proposition 8.9]{DWZ1}; rigidity follows from it, and nondegeneracy is \cite[Proposition 2.3]{GLFS}.
\end{proof}

Let $\mc C$ be a Hom-finite Krull--Schmidt $2$-Calabi--Yau triangulated category and let $U$ be rigid.
Set
\[ \mc Z_U=U[-1]^\perp={}^\perp U[1], \qquad \mc C_U=\mc Z_U/[U]. \]
By Iyama--Yoshino reduction, $\mc C_U$ is again Hom-finite and $2$-Calabi--Yau, and cluster-tilting objects of $\mc C$ containing $U$ correspond to cluster-tilting objects of $\mc C_U$ \cite[Theorems 4.7 and 4.9]{IY}.
We write $\Sigma_U$ for its suspension and $\br X$ for the image of $X\in\mc Z_U$.
Thus
\[ \End_{\mc C_U}(\br X)=\End_{\mc C}(X)/[U](X,X). \]

For later use, if $X=U\oplus V$ and $p=1_U\in\End(X)$, then the block decomposition of $\End(X)$ gives the canonical algebra isomorphism
\begin{equation}\label{eq:idempotent-quotient} \End(X)/\End(X)p\End(X) \cong \End(V)/[U](V,V). \end{equation}
Thus deleting the idempotents belonging to $U$ is the same as quotienting the remaining endomorphisms by those factoring through $\op{add}(U)$.

Return to $\mc C=\CQ$.
Let $e$ be rigid, let $U$ be its corresponding object, and let $V^\pm$ correspond to the positive and negative complements $e_c^\pm$.
Then $U\oplus V^\pm$ are cluster-tilting and $V^\pm\in\mc Z_U$.
Put
\[ E=\coker(e), \qquad C=\coker(e_c^+), \qquad P_e=f_\ep(C). \]
By Corollary \ref{C:indP}, $P_e$ is a basic projective generator of $\rep(e^\perp)$.

\begin{proposition}\label{P:additive-reduction}
The functor $L_e$ induces an additive equivalence
\[ \op{add}(C)/[E] \xrightarrow{\ \sim\ } \op{proj}\rep(e^\perp). \]
Equivalently, for $C_1,C_2\in\op{add}(C)$,
\[ \Hom_J(C_1,C_2)/[E](C_1,C_2) \cong \Hom_{\rep(e^\perp)}(L_eC_1,L_eC_2). \]
\end{proposition}
\begin{proof}
The proof of Lemma \ref{L:factor} applies to finite direct sums of summands of $C$; taking the corresponding off-diagonal blocks gives the Hom-isomorphism.
Since $L_e(C)=P_e$ is a basic projective generator, the induced functor is essentially surjective onto the projectives.
\end{proof}

\begin{corollary}\label{C:reduction-algebra}
There are algebra isomorphisms
\[ \End_{(\CQ)_U}(\br{V^+}) \cong \End_J(C)/[E](C,C) \cong \End_{\rep(e^\perp)}(P_e) \cong {_e}A. \]
Moreover,
\[ \br{V^-}\cong\Sigma_U\br{V^+}, \]
so the same reduction algebra is obtained from either completion.
\end{corollary}
\begin{proof}
Jasso's compatibility between $2$-Calabi--Yau and $\tau$-tilting reduction identifies the first two algebras \cite[Theorem 4.23(a)]{Ja}; shifted-projective summands of $U$ are allowed as noted after \cite[Setting 4.19]{Ja}.
Proposition \ref{P:additive-reduction} and Corollary \ref{C:indP} give the remaining identifications.
The commutative reduction diagram in \cite[Theorem 4.23(b)]{Ja} sends the positive and negative completions to $P_e$ and $P_e[1]$, respectively, whence
$\br{V^-}\cong\Sigma_U\br{V^+}$.
\end{proof}

\subsection{The projected quiver with potential}
Let $e$ be rigid. We use the preceding notation.
Corollary \ref{C:perpQP-abstract} gives a Jacobian realization for every rigid presentation. Under extended reachability, the following theorem identifies a projected QP directly from $(Q,\S)$.

\begin{corollary}[Jacobian realization without reachability]\label{C:perpQP-abstract}
For every rigid presentation $e$ there exists a Jacobi-finite complete QP $(Q_e,\S_e)$, possibly with loops or $2$-cycles, such that
\[ J(Q_e,\S_e) \cong \End_{\rep(e^\perp)}(P_e) \cong {_e}A. \]
Consequently,
\[ \rep(Q_e,\S_e)\simeq\rep(e^\perp). \]
\end{corollary}
\begin{proof}
Apply Proposition \ref{P:KLend} to $U\oplus V^+$ and restrict the resulting QP to the vertices labelled by $V^+$.
By Lemma \ref{L:epi}, the restricted Jacobian algebra is the quotient by the idempotent $1_U$; using \eqref{eq:idempotent-quotient} and Corollary \ref{C:reduction-algebra} gives
\[ J(Q_e,\S_e) \cong \End_{\CQ}(V^+)/[U](V^+,V^+) \cong {_e}A. \]
\end{proof}
\noindent This QP is noncanonical; it need not be mutation-equivalent to a restriction of $(Q,\S)$, rigid, or nondegenerate.

For $\varepsilon\in\{+,-\}$, suppose that the completion
$e^\varepsilon=e\oplus e_c^\varepsilon$ is extended-reachable.
Let $I_\varepsilon$ be the vertices of $\mu_{e^\varepsilon}(Q,\S)$ corresponding to $e_c^\varepsilon$ and define
\[ {_e(Q,\S)}^\varepsilon :=\mu_{e^\varepsilon}(Q,\S)|_{I_\varepsilon}. \]
When a sign is fixed, we abbreviate it to $\eQS$.

\begin{theorem}[Explicit projected QP]\label{T:perpQP}
If $e^\varepsilon$ is extended-reachable, then
\[ J({_e(Q,\S)}^\varepsilon) \cong \End_{\rep(e^\perp)}(P_e) \cong {_e}A, \qquad \rep({_e(Q,\S)}^\varepsilon)\simeq\rep(e^\perp). \]
If both completions are extended-reachable, their explicit projected Jacobian algebras are isomorphic.
Rigidity and nondegeneracy of $(Q,\S)$ pass to every explicit projected QP.
\end{theorem}
\begin{proof}
Proposition \ref{P:QPend} gives the idempotent-compatible isomorphism
\[ J(\mu_{e^\varepsilon}(Q,\S)) \cong \End_{\CQ}(U\oplus V^\varepsilon). \]
Restriction, \eqref{eq:idempotent-quotient}, and Corollary \ref{C:reduction-algebra} give
\begin{align*}
 J({_e(Q,\S)}^\varepsilon)
 &\cong
 \End_{\CQ}(U\oplus V^\varepsilon)/\innerprod{1_U}\\
 &\cong
 \End_{(\CQ)_U}(\br{V^\varepsilon})
 \cong {_e}A.
\end{align*}
The two signs therefore yield isomorphic Jacobian algebras.
Mutation preserves rigidity and nondegeneracy, and Lemma \ref{L:epi} passes these properties to the restriction.
\end{proof}

\begin{proposition}\label{P:Breduction}
For every basic rigid $e$, the antisymmetrized Euler form of $(\CQ)_U$, in the simple basis dual to $\br{V^+}$, has matrix
\begin{equation}\label{eq:Bred-categorical} B_e =\C_{e^\perp}^{\top}B(Q)\C_{e^\perp}. \end{equation}
If $e^\varepsilon$ is extended-reachable, this is the ordinary exchange matrix of the explicit projected quiver ${_eQ}^\varepsilon$.
\end{proposition}
\begin{proof}
By Remark \ref{P:Cred-general}, the columns of $\C_{e^\perp}$ are the classes of the reduced simple objects in the basis of the canonical simples of $J(Q,\S)$.
Palu's antisymmetrized Euler-form formula is bilinear on $K_0$, so its matrix in the reduced simple basis is the change-of-basis matrix in \eqref{eq:Bred-categorical} \cite[Section 3]{Pa}.
For an explicit projected QP, Theorem \ref{T:perpQP} identifies these simples with the vertex simples; the same form is therefore its quiver exchange matrix.
\end{proof}

\section{Initial-seed mutation formulas} \label{S:inmu}
This section concerns the projected QP obtained by mutation and restriction. Accordingly, all statements involving that QP are subject to the reachability hypotheses stated below; Corollary \ref{C:perpQP-abstract} is not used in the mutation calculations.

Two mutation operations occur below. Initial-seed mutation $\mu_k$ changes the ambient initial seed while transporting a fixed cluster, whereas $\sigma_k$ exchanges the $k$-th summand of a fixed cluster. We record the corresponding $C$-matrix formulas.

Let $d=\bigoplus_i d_i$ be a cluster. Recall that $\Delta_d$ has the weights of the $d_i$ as rows and that $\C_d=\Delta_d^{-1}$. Put
\[ B_d:=\C_d^\top B\C_d, \qquad \C_d=(\gamma_1,\dots,\gamma_n), \]
and let $b_k$ be the $k$-th row of $B=B(Q)$. Every column of $\Delta_d$ is sign-coherent \cite{DWZ2,DF}; write $\sgn(k,d)$ for the sign of its $k$-th column. This differs from $\sgnc(k,d)$, the sign of the $k$-th column of $\C_d$ in Theorem \ref{T:Cmat}. Let $J_k$ be the identity matrix with its $(k,k)$-entry replaced by $-1$, and let $A^{k\bullet}$ denote the matrix obtained from $A$ by setting all rows except the $k$-th to zero.
We set 
$$E_k(B,d):=J_k+[\sgn(k,d)B]_+^{k\bullet}.$$

\begin{lemma}[{\cite[(5.9)]{FZ4}}]\label{L:Cmu}
	If $d$ is an extended-reachable cluster, then under initial-seed mutation at $k$,
	\[ \C_{\mu_k(d)}=E_k(B,d)\C_d. \]
	In particular, only row $k$ changes, and
	\[ \mu_k(\gamma_i)(k)=[\sgn(k,d)b_k]_+\gamma_i-\gamma_i(k). \]
\end{lemma}
\begin{proof}
	This is the inverse matrix transformation of the usual initial-seed mutation formula for weights in $\C_d\Delta_d=I$.
\end{proof}

If $B'$ is the mutated initial exchange matrix, then
\[ B'=E_k(B,d)^\top BE_k(B,d), \qquad E_k(B,d)^2=I. \]
Consequently,
\begin{equation}\label{eq:Bmu}
	B_{\mu_k(d)}=\C_{\mu_k(d)}^\top B'\C_{\mu_k(d)}=B_d.
\end{equation}
Set
\[ R_k(d):=J_k+[-\sgnc(k,d)B_d]_+^{k\bullet}. \]
For an extended-reachable cluster $d$, the $C$-matrix exchange formula \cite[Proposition 1.3]{NZ} is
\begin{equation}\label{eq:exC}
	\C_{\sigma_k(d)}=\C_dR_k(d).
\end{equation}

\subsection{Mutation of complements} \label{ss:muDelta}
Throughout this subsection, $e=\bigoplus_{a=1}^r e_a$ is a basic rigid presentation, with total weight $\ep=\sum_a\ep_a$. Write
\[ e^\pm=e\oplus e_c^\pm, \qquad e_c^\pm=\bigoplus_{i=1}^{n-r}e_i^\pm, \]
and order the complementary summands first. By \cite[Theorems 4.5 and 4.6]{CGY}, a cluster containing $e$ is its positive, respectively negative, completion precisely when the columns indexed by the complementary summands are nonnegative, respectively nonpositive.

\begin{lemma}\label{L:existj}
	If $\ep(k)=0$, then the $k$-th column of $\Delta_{e^\pm}$ has zero entries on every summand of $e$ and equals $\pm\eu_j$ on the complementary block, for a unique $j\leq n-r$. Equivalently, the $j$-th column of $\pm\C_{e^\pm}$ is $\eu_k$.
\end{lemma}
\begin{proof}
	Fix $\varepsilon\in\{+,-\}$. Sign-coherence and $\sum_a\ep_a(k)=0$ give $\ep_a(k)=0$ for every $a$. Let $S$ be the first $n-r$ columns of $\varepsilon\C_{e^\varepsilon}$ and let $x$ be the complementary block of the $k$-th column of $\varepsilon\Delta_{e^\varepsilon}$. Then $Sx=\eu_k$. By Remark \ref{P:Cred-general}, $S$ is integral and nonnegative, while sign-coherence forces $x\geq0$. Since the columns of $S$ are nonzero, this equality gives $x=\eu_j$ and $S(-,j)=\eu_k$ for a unique $j$.
\end{proof}

\begin{theorem}[Mutation of signed complements]\label{T:mucomp}
	Fix $\varepsilon\in\{+,-\}$. Suppose that $e^\varepsilon$ is extended-reachable. If $\ep(k)=0$, let $j\leq n-r$ be the index in Lemma \ref{L:existj}. Then
	\begin{equation}\label{eq:mucomp-full}
		(\mu_k e)^\varepsilon=
		\begin{cases}
			\mu_k(e^\varepsilon),&\ep(k)\neq0,\\
			\sigma_j\mu_k(e^\varepsilon),&\ep(k)=0.
		\end{cases}
	\end{equation}
\end{theorem}
\begin{proof}
	Write $\C_{e^\varepsilon}=(\gamma_1,\ldots,\gamma_n)$ and $\bar\gamma_i=\varepsilon\gamma_i$. The complementary columns $\bar\gamma_i$ are nonnegative, and this sign condition characterizes the $\varepsilon$-completion. Since both candidates in \eqref{eq:mucomp-full} contain $\mu_k(e)$, it remains to check their complementary columns.
	
	If $\ep(k)\neq0$, sign-coherence gives $\sgn(k,e^\varepsilon)=\sgn(\ep(k))$. Lemma \ref{L:Cmu} makes the new $k$-th entry nonnegative when $\bar\gamma_i(k)=0$. If $\bar\gamma_i(k)>0$, then $\bar\gamma_i$ has a positive entry outside $k$, since otherwise $\ep_a\bar\gamma_i=0$ for every $a$ would force $\ep(k)=0$. That entry is unchanged, so sign-coherence forces the mutated column to have sign $\varepsilon$.
	
	If $\ep(k)=0$, Lemma \ref{L:existj} gives $\gamma_j=\varepsilon\eu_k$. The $j$-th exchange exists by Theorem~\ref{T:+-}: deleting this summand from $\mu_k(e^\varepsilon)$ leaves an almost complete rigid presentation. Initial-seed mutation changes it to $-\varepsilon\eu_k$, and the $j$-th exchange changes it back. Every other complementary column has a positive entry outside $k$, since otherwise it would be a positive multiple of $\gamma_j$. Both operations change only row $k$, so that entry survives; sign-coherence again forces sign $\varepsilon$. The completion criterion proves \eqref{eq:mucomp-full}.
\end{proof}

A basic rigid presentation $e$ is \emph{$\varepsilon$-reachable} if an ordinary mutation sequence sends it to $\varepsilon P_I$, where
\[ P_I:=\bigoplus_{i\in I}P_i, \qquad \varepsilon P_I=P_I\text{ for }\varepsilon=+, \quad \varepsilon P_I=P_I[1]\text{ for }\varepsilon=-, \]
for some $r$-element subset $I\subseteq Q_0$.

\begin{corollary}\label{C:reachable-complement}
If $e$ is $\varepsilon$-reachable, then $e^\varepsilon$ is extended-reachable.
\end{corollary}
\begin{proof}
Start at $\varepsilon(P_{I^c}\oplus P_I)$ and read a reachability
sequence for $e$ backward. If the relevant weight has zero
$k$-th coordinate, first exchange the complementary summand $j$
of Lemma~\ref{L:existj}, then reverse the initial-seed mutation at
$k$; otherwise reverse that mutation directly. Each complementary
exchange exists by the two-complement theorem (Theorem~\ref{T:+-}).
Theorem~\ref{T:mucomp} identifies the resulting clusters as the
successive $\varepsilon$-completions. The QP mutations in this
construction are subject to the standing admissibility convention.
For $\varepsilon=+$, a global shift takes the positive projective
cluster to the negative one.
\end{proof}

\begin{algorithm}[Construction of the explicit projected QP]\label{a:Qep}
Given $\mub(e)=\varepsilon P_I$, start at
$\varepsilon(P_{I^c}\oplus P_I)$ and reverse the sequence as in
Corollary~\ref{C:reachable-complement}. This constructs $e^\varepsilon$
and a sequence taking it to the signed projective cluster (with a
global shift when $\varepsilon=+$).
Apply the resulting sequence to $(Q,\S)$ and delete the vertices corresponding to the summands $P_I$. By Theorem \ref{T:perpQP}, the result is the explicit projected QP for $\rep(e^\perp)$.
\end{algorithm}

Corollary \ref{C:perpQP-abstract} requires no reachability hypothesis, but it does not describe the QP in terms of $(Q,\S)$. If only the projected quiver is needed, Algorithm \ref{a:Qepsimple} avoids tracking the potential and the complementary exchanges. The procedure in Algorithm \ref{a:Qep} can also be carried out for Jacobi-infinite QPs; whether it still yields a QP model of the orthogonal category remains open.

\subsection{Mutation of simples in $\rep(e^\perp)$} \label{ss:muSperp}
Theorem \ref{T:mucomp} gives the following piecewise-linear rule for the completion $C$-matrix. Write $\max_+=\max$ and $\max_-=\min$.

\begin{theorem}[Mutation of simples]\label{T:musimple}
	Let $e=\bigoplus_{a=1}^r e_a$ be basic rigid and order $e^\pm$ as above. Fix $\varepsilon\in\{+,-\}$ and suppose that $e^\varepsilon$ is extended-reachable; the formula below uses this sign. The matrix $\C_{\mu_k(e)^\pm}$ differs from $\C_{e^\pm}$ only in row $k$. If $\gamma_i'$ is its $i$-th column, then
	\[
	\gamma_i'(k)=
	\begin{cases}
		[\sgn(\ep(k))b_k]_+\gamma_i-\gamma_i(k),&\ep(k)\neq0,\\
		\max_{\mp}\!\bigl([-b_k]_+\gamma_i,[b_k]_+\gamma_i\bigr)-\gamma_i(k),&\ep(k)=0\text{ and }i\neq j,\\
		\pm1,&\ep(k)=0\text{ and }i=j,
	\end{cases}
	\]
	where $j\leq n-r$ is the index in Lemma \ref{L:existj}.
\end{theorem}
\begin{proof}
	Fix $\varepsilon\in\{+,-\}$. If $\ep(k)\neq0$, Theorem \ref{T:mucomp} requires no complementary exchange, so Lemma \ref{L:Cmu} gives the first branch.
	
	Assume $\ep(k)=0$ and put $d^\circ=\mu_k(e^\varepsilon)$. Lemma \ref{L:existj} gives $\gamma_j=\varepsilon\eu_k$, hence
	\[ \gamma_j^\circ=-\varepsilon\eu_k, \qquad \gamma_i^\circ(k)=[\varepsilon b_k]_+\gamma_i-\gamma_i(k) \quad(i\neq j). \]
	The $j$-th exchange returns $\gamma_j^\circ$ to $\varepsilon\eu_k$. Moreover, the one-step invariance \eqref{eq:Bmu} gives $B_{d^\circ}=B_{e^\varepsilon}$, so
	\[ B_{d^\circ}(j,i)=\gamma_j^\top B\gamma_i=\varepsilon b_k\gamma_i. \]
	Since $\sgnc(j,d^\circ)=-\varepsilon$, the exchange coefficient is $[b_k\gamma_i]_+$. Equation \eqref{eq:exC} therefore gives
	\[ \gamma_i'(k)=[\varepsilon b_k]_+\gamma_i-\gamma_i(k)-\varepsilon[b_k\gamma_i]_+. \]
	Writing $a=[b_k]_+\gamma_i$ and $c=[-b_k]_+\gamma_i$ turns this into $\min(a,c)-\gamma_i(k)$ for $\varepsilon=+$ and $\max(a,c)-\gamma_i(k)$ for $\varepsilon=-$. Both operations affect only row $k$.
\end{proof}

\begin{algorithm}[Simplified construction of the projected quiver]\label{a:Qepsimple}
Use the reachability sequence and reverse construction in Algorithm~\ref{a:Qep}.
At the terminal cluster $\varepsilon(P_{I^c}\oplus P_I)$, start with
the $C$-matrix $\varepsilon I$. Reading the sequence backward,
update the exchange matrix and apply Theorem~\ref{T:musimple}.
The first $n-r$ complementary columns of the resulting matrix,
multiplied by $\varepsilon$, form $\C_{e^\perp}$.
By Proposition \ref{P:Breduction}, the categorical reduced exchange matrix is
	\[ \C_{e^\perp}^\top B\C_{e^\perp}; \]
	when the explicit projected QP is defined, this is its quiver matrix.
\end{algorithm}

\begin{example}
	Consider the following quiver with a nondegenerate potential:
	\[ \Markovext \]
	Let $\ep=(0,0,1,-2)$. Starting from $-\eu_1$, the mutations $\mu_1$ and $\mu_3$ produce $\ep$; equivalently,
	\[ (-1,0,0,0)\xrightarrow{\mu_1}(1,0,-1,0)\xrightarrow{\mu_3}(0,0,1,-2). \]
	Since both mutated coordinates are nonzero, we display $-\C_{e^-}$. Theorem \ref{T:musimple} gives
	\[
	\sm{1&0&0&0\\0&1&0&0\\0&0&1&0\\0&0&0&1}
	\xrightarrow{\mu_1}
	\sm{-1&0&1&0\\0&1&0&0\\0&0&1&0\\0&0&0&1}
	\xrightarrow{\mu_3}
	\sm{-1&0&1&0\\0&1&0&0\\-1&0&0&2\\0&0&0&1}.
	\]
	Hence
	\[
	\C_{e^\perp}^\top=\sm{0&1&0&0\\1&0&0&0\\0&0&2&1},
	\qquad
	\C_{e^\perp}^\top B\C_{e^\perp}=\sm{0&-2&2\\2&0&-2\\-2&2&0},
	\]
	the Markov quiver.
	
	For comparison, Theorem \ref{T:mucomp} gives
	$e_c^-=-\eu_4\oplus-\eu_2\oplus(-1,0,1,-2)$.
	The sequence $\mu_1\mu_3$ sends $e\oplus e_c^-$ to the negative cluster and sends the original exchange matrix to
	\[ B'=\sm{0&2&-2&1\\-2&0&2&-1\\2&-2&0&0\\-1&1&0&0}. \]
	Deleting the last vertex gives the same projected quiver.
\end{example}

\subsection{Schur reductions and exchange pairs}
Recall the Schur-reduction vectors $\gamma_\pm(e)$ of Definition~\ref{D:Schredbr}, and their completion-matrix formulas.

\begin{corollary}\label{C:muSchred}
Let $e$ be indecomposable rigid. If $e^\pm$ is extended-reachable, the following mutation rule holds for $\gamma_\pm$. Let $\gamma_\pm'=\gamma_\pm(\mu_k(e))$. Only the $k$-th coordinate changes, and
	\begin{equation}\label{eq:muSchred}
		\gamma_\pm'(k)=
		\begin{cases}
			[\sgn(\ep(k))b_k]_+\gamma_\pm-\gamma_\pm(k),&\ep(k)\neq0,\\
			\max\!\bigl([-b_k]_+\gamma_\pm,[b_k]_+\gamma_\pm\bigr)-\gamma_\pm(k),&\ep(k)=0.
		\end{cases}
	\end{equation}
\end{corollary}
\begin{proof}
	By the $C$-matrix descriptions of $\gamma_\pm(e)$, Theorem~\ref{T:musimple} gives the positive formula from the last column of $\C_{e^-}$. The negative formula follows from the last column of $-\C_{e^+}$; negation changes the minimum in the zero-coordinate branch into a maximum.
\end{proof}

\begin{conjecture}
	Let $d$ be an indecomposable general presentation
	such that $\mu_k(d)$ is also indecomposable and general.
	Then $\gamma_\pm(\mu_k(d))$ are obtained from
	$\gamma_\pm(d)$ by the mutation rule
	\eqref{eq:muSchred}, with $\ep$ replaced by
	the weight $\delta$ of $d$.
\end{conjecture}

\begin{proposition}\label{P:regular-Schur}
Let $e$ be indecomposable rigid, and let $e_\pm$ be the other
complements of $e_c^\pm$. Set $E=\coker(e)$
and $E_\pm=\coker(e_\pm)$.
Let $S_-$ and $S_+$ be the unique simple objects of
$\rep((e_c^-)^\perp)$ and $\rep((e_c^+)^\perp)$, respectively.
Then 
\[
\begin{aligned}
(e_-,e)\text{ regular}&\iff
\Ext_A^1(S_-,S_-)=0,\\
(e,e_+)\text{ regular}&\iff
\Ext_A^1(S_+,S_+)=0.
\end{aligned}
\]
For regular pairs with $E\neq0$ or $\tau E\neq0$, respectively,
we also have
\[
\begin{aligned}
 \hom(e,\tau e_-)&=1,
 &\rho_+(e)&\cong\Img(E\to\tau E_-),\\
 \hom(e_+,\tau e)&=1,
 &\rho_-(\tau e)&\cong\Img(E_+\to \tau E),
\end{aligned}
\]
In particular, the first pair is regular if $E$ is Schur, and the second if $\tau E$ is Schur.
\end{proposition}
\begin{proof}
Let $e_c$ be either $e_c^-$ or $e_c^+$, and put $B={}_{e_c}A$.
Take $(d_-,d_+)=(e_-,e)$ in the first case and $(e,e_+)$ in the
second. The perpendicular category $\rep(e_c^\perp)$ has a unique simple $S$.
Theorem \ref{T:bijection} and Corollary \ref{C:isoE}, with Remark
\ref{r:decomposable}, identify the exchange pair with $(B[1],B)$;
hence
\[ \e_A(d_-,d_+)=\e_B(B[1],B)=\dim_K B. \]
Since $B$ is basic local and $\rep(e_c^\perp)$ is full and extension
closed,
\[ \dim_K B=1 \iff \Ext_B^1(S,S)=0 \iff \Ext_A^1(S,S)=0. \]
The Hom-dimensions and image descriptions follow from
\cite[Theorem~7.6]{Ftf} in the corresponding nonzero regular cases.
If $E$ (respectively $\tau E$) is Schur, its indicated Schur
reduction is itself and is rigid, so the corresponding Ext-group
vanishes.
\end{proof}

For a regular exchange pair $(d_-,d_+)$ with common complement $d_0$, let $L=\coker(d_+)$, $N=\coker(\tau d_-)$, and let $f:L\to N$ be the unique nonzero homomorphism. Write $\delta_\pm$ for the weights of $d_\pm$. By \cite[Theorem 7.6]{Ftf}, the corresponding $c$-vector is $\pm\dv\Img(f)$, where
\[ \Img(f)=\fc_{\delta_-}(L)=t_{\delta_+}(N). \]
Moreover, \cite[Theorem 3.4 and Lemma 3.7]{Fcf} show that $\dv\Img(f)$ is a vertex of the Newton polytopes $\Nc(L)$ and $\N(N)$. Recall that $\N(M)$ is the convex hull of the dimension vectors of subrepresentations of $M$, while $\Nc(M)$ is the convex hull of the dimension vectors of its quotients.

Since $\E(d_+,\Img(f))=0$ and $\hom(L,\Img(f))=1$, we have $\delta_+\cdot\dv\Img(f)=1$.
\begin{question}\label{q:c1}
Must every vertex $\gamma$ of $\Nc(L)$ satisfying
$\delta_+\cdot\gamma=1$ be a positive $c$-vector of some cluster?
\end{question}

\section{Modified projections for quivers with potentials} \label{S:OPmod}
Throughout this section, $e$ is an arbitrary basic rigid presentation and $U$ is the corresponding rigid object of $\CQ$. The constructions below are additive functors on the cluster category; Plamondon's correspondence induces the associated maps on decorated representations.
Whenever an explicit projected QP is used, the relevant completion is assumed extended-reachable, and a mutation sequence taking it to a signed projective cluster is fixed.

\subsection{Adjoints for Calabi--Yau reduction}
Let
\[ q_U:\CQ\longrightarrow \CQ/[U], \qquad \iota_U:(\CQ)_U=\mc Z_U/[U]\longrightarrow\CQ/[U] \]
be the quotient and the fully faithful inclusion.

\begin{proposition}\label{P:reduction-adjoints}
	The functor $\iota_U$ has a left adjoint $\pi_U^-$ and a right adjoint $\pi_U^+$.
	Write $U=\bigoplus_{a=1}^rU_a$ and set
	\[ U_X^-:=\bigoplus_{a=1}^r \Hom_{\CQ}(U_a[-1],X)\otimes_K U_a[-1], \qquad U_X^+:=\bigoplus_{a=1}^r \Hom_{\CQ}(X,U_a[1])^*\otimes_K U_a. \]
	For $X\in\CQ$, the values of the adjoints are represented by the universal approximation triangles
	\begin{align}
		U_X^-&\longrightarrow X\longrightarrow X_e^-\longrightarrow U_X^-[1],
		\label{eq:sqtri-}\\
		U_X^+&\longrightarrow X_e^+\longrightarrow X\longrightarrow U_X^+[1].
		\label{eq:sqtri+}
	\end{align}
	Moreover,
	\[ \pi_U^-(q_UX)=\br{X_e^-}, \qquad \pi_U^+(q_UX)=\br{X_e^+}. \]
	If $\Phi:\CQ\to\mc C'$ is a triangle equivalence, then it induces an equivalence
	\[ \Phi_U:(\CQ)_U\xrightarrow{\sim}\mc C'_{\Phi U} \]
	and canonical natural isomorphisms
	\[ \Phi_U\pi_U^\pm\cong\pi_{\Phi U}^\pm\bar\Phi, \]
	where $\bar\Phi$ is the equivalence induced on the additive quotients.
\end{proposition}
\begin{proof}
	The universal maps give $X_e^\pm\in\mc Z_U$.
	For $Z\in\mc Z_U$, applying $\Hom_{\CQ}(-,Z)$ to \eqref{eq:sqtri-} and passing modulo maps through $U$ gives
	\[ \Hom_{(\CQ)_U}(\br{X_e^-},\br Z) \cong \Hom_{\CQ/[U]}(q_UX,q_UZ). \]
	The dual calculation with \eqref{eq:sqtri+} gives the right adjoint.
	Adjoints are unique up to canonical natural isomorphism, so the constructions are functorial and commute with triangle equivalences.
\end{proof}

\begin{definition}\label{D:sqproj}
	Let $d$ correspond to $X\in\CQ$, and choose the approximation triangles in Proposition \ref{P:reduction-adjoints}.
	Set
	\[ \wtd{\sqcup}_e^\pm(d):=\wtd F_{Q,\S}(X_e^\pm), \]
	and let $\sqcup_e^\pm(d)$ be obtained by deleting all direct summands in $\op{add}(e)$.
	Equivalently, $\sqcup_e^\pm(d)$ is the decorated representation of a reduced representative of $\pi_U^\pm(q_UX)$.
We call $\sqcup_e^+$ and $\sqcup_e^-$ the positive and negative $\E$-orthogonalizations relative to $e$; the maps $\wtd{\sqcup}_e^\pm$ are their unreduced versions.
\end{definition}

Proposition \ref{P:reduction-adjoints} shows that these maps are well-defined on isomorphism classes and restrict to the identity on $\drep(\comp e)$. Moreover, for every extended mutation sequence $\mub$,
\begin{equation}\label{eq:sqmu} \mub\bigl(\sqcup_e^\pm(d)\bigr)=\sqcup_{\mub(e)}^\pm\bigl(\mub(d)\bigr). \end{equation}
The same statements hold for $\wtd{\sqcup}_e^\pm$.

\begin{lemma}\label{L:approxtrunc}
	Let $e,d$ correspond to $U,X$.
	\begin{enumerate}
		\item If $\E(d,e)=0$, then
		$\wtd{\sqcup}_e^-(d)\cong\eot e(d) \text{ and }  \sqcup_e^-(d)\cong\eo e(d).$
		\item If $\E(e,d)=0$, then
		$\wtd{\sqcup}_e^+(d)\cong\eort e(d) \text{ and } \sqcup_e^+(d)\cong\eor e(d).$
	\end{enumerate}
\end{lemma}
\begin{proof}
	By \cite[Proposition 3.10]{P}, $\E(e,d)$ is the subspace of $\Hom_{\CQ}(U[-1],X)$ formed by morphisms factoring through the canonical cluster-tilting object, while
	\[ \dim\Hom_{\CQ}(U[-1],X)=\e(e,d)+\e(d,e). \]
	Under the first vanishing hypothesis, that subspace is the entire Hom-space, so the truncation triangle defining $\eot e(d)$ is isomorphic to \eqref{eq:sqtri-}.
	The second statement is dual, and reduction removes the $e$-summands.
\end{proof}

By Proposition \ref{P:reduction-adjoints}, if $\mub$ is an extended mutation sequence and $e'=\mub(e)$, then
\[ \wtd{\sqcup}_e^-(d) \cong\mub^{-1}\eot{e'}\mub(d) \quad\text{whenever }\E(\mub(d),e')=0, \]
and dually
\[ \wtd{\sqcup}_e^+(d) \cong\mub^{-1}\eort{e'}\mub(d) \quad\text{whenever }\E(e',\mub(d))=0. \]
If $e$ is extended-reachable and $\mu_-(e)=P_-[1]$, $\mu_+(e)=P_+$, the required vanishing is automatic, and
\[ \wtd{\sqcup}_e^-=\mu_-^{-1}\eot{P_-[1]}\mu_-, \qquad \wtd{\sqcup}_e^+=\mu_+^{-1}\eort{P_+}\mu_+. \]
Thus Definition \ref{D:sqproj} recovers the original mutation-based construction whenever that construction is available.

\begin{remark} The maps $\sqcup_e^\pm$ and $\mc I_e\L_e$ have the same target $\drep(\comp e)$ but need not agree; the analogous statement holds for the dual maps.
	We do not know an analogue of $\sqcup_e^\pm$ for non-Jacobian algebras.
	
The following noncommutative diagram summarizes most functors we encountered so far.
Although $t_\epc$ and $f_\ep$ are defined on $\rep(Q, S)$, they behave particularly nice when restricted
to these subcategories.
	$$\xymatrix{	
		e^\perp \ar@<.5ex>[rr]^{\tau} \ar@<.5ex>[dd]^{\mc I_e} && {^\perp \nu e} \ar@<.5ex>[ll]^{\tau^{-1}} \ar@<.5ex>[dd]^{\check{\mc I}_{\nu e}}	\\
		& \drep(Q,\S) \ar[ul]^{\L_e}_{\heo{\ep}} \ar[ur]_{\mc{R}_{\nu e}}^{\perp_{\nu \ep}} \ar[dr]_{\sqcup_{\nu e}^+}^{\sqcup_{\nu e}^-} \ar[dl]^{\sqcup_e^+}_{\sqcup_e^-} && \\
		\comp{e} \ar@<.5ex>[uu]^{f_\ep} \ar@<.5ex>[rr]^{\tau} && \comp{\nu e} \ar@<.5ex>[ll]^{\tau^{-1}} \ar@<.5ex>[uu]^{t_{\epc}}	
	} 
	$$
\end{remark}

\begin{proposition}[Compatibility with Calabi--Yau reduction]\label{L:CY-Lcompare}
Let $T_e=\br{V^+}$ be the cluster-tilting object of $(\CQ)_U$ and
let $\wtd F_{T_e}$ denote its decorated-representation correspondence.
For every $X\in\mc Z_U$, without a rigidity assumption on $X$,
\[ \wtd F_{T_e}(\br X)\cong \L_e\bigl(\wtd F_{Q,\S}(X)\bigr) \qquad\text{in }\drep(e^\perp), \]
where the target is identified using
$\End_{(\CQ)_U}(T_e)\cong{}_eA$ from
Corollary \ref{C:reduction-algebra}.
\end{proposition}
\begin{proof}
Let $d=\wtd F_{Q,\S}(X)$ and $M=F_{Q,\S}(X)$.
The presentation $d$ is compatible with $e$, so $f_\ep M=L_e(M)$.
With $C=\coker(e_c^+)$ and $P_e=L_e(C)$, Jasso's comparison
\cite[Proposition~4.25]{Ja} and adjunction give
\[ \Hom_{(\CQ)_U}(T_e,\br X) \cong\Hom_A(C,f_\ep M) \cong\Hom_{\rep(e^\perp)}(P_e,L_e(M)). \]
This identifies the ordinary module components under the above
Morita equivalence, for all $X\in\mc Z_U$.

For the decorations, an indecomposable $X\notin\op{add}(U)$
remains indecomposable in $\mc Z_U/[U]$: its endomorphism ring
there is a nonzero quotient of the local ring $\End_{\CQ}(X)$.
By Corollary \ref{C:reduction-algebra},
$\br{V^-}\cong\Sigma_U T_e$, so the summands corresponding to
$e_i^-$ give precisely the negative projectives $P_{i,e^\perp}[1]$.
Every other compatible indecomposable outside $\op{add}(e)$ has
nonzero projected module part by Lemma \ref{L:0fibre}; as its
image is indecomposable, it has no negative decoration.
Objects of $\op{add}(U)$ are killed by both maps.
The assertion follows by additivity and Definition \ref{D:Lext}.
\end{proof}

\subsection{Projection to an explicit projected QP}
Fix a mutation sequence taking $e^\varepsilon$ to $\varepsilon(P_{I^c}\oplus P_I)$ and defining the explicit projected QP ${_e(Q,\S)}^\varepsilon$. Let $\Phi$ be the associated Keller--Yang equivalence.
Proposition \ref{P:reduction-adjoints} induces
\[ \Phi_U:(\CQ)_U\xrightarrow{\sim}(\CQ')_{\varepsilon P_I}. \]
Set $T_e=\overline{V^+}$. By Corollary
\ref{C:reduction-algebra} and Theorem \ref{T:perpQP},
there are idempotent-compatible algebra isomorphisms
\[\operatorname{End}_{(\mathcal C_{Q,\mathcal S})_U}(T_e)\cong\operatorname{End}_{\operatorname{rep}(e^\perp)}(P_e)
\cong J({}_e(Q,\mathcal S)^\varepsilon).\]
For $\varepsilon=-$, we use $\overline{V^-}\cong\Sigma_U T_e$
to identify the corresponding cluster-tilting objects.
This induces a bijection
\[ \lambda_e^\varepsilon: \drep(e^\perp)\xrightarrow{\sim} \drep({_e(Q,\S)}^\varepsilon), \]
which we also use on ordinary representations.
The map $\lambda_e^\varepsilon$ is induced by the same reduction equivalence $\Phi_U$.

\begin{definition}\label{D:Lpm}
	Define
	\[ \L_e^\varepsilon:=\L_e\circ\sqcup_e^\varepsilon. \]
	When ${_e(Q,\S)}^\varepsilon$ is fixed, define the modified projection to this QP by
	\[ \L_{(e)}^\varepsilon :=\lambda_e^\varepsilon\circ\L_e^\varepsilon. \]
	Dually, set
	\[ \mc R_{\ec}^\pm:=\mc R_{\ec}\circ\sqcup_{\ec}^\pm. \]
\end{definition}

For a representation $M$, we set
\[ L_e^\pm(M):=\coker\L_e^\pm(d_M), \qquad R_{\ec}^\pm(M):=\ker\mc R_{\ec}^\pm(\dc_M). \]
By Lemma \ref{L:Lext}.(2),
\begin{equation}\label{eq:Lepm} L_e^\pm(M)=L_e\!\left(\coker(\sqcup_e^\pm(d_M))\right)=f_\ep\!\left(\coker(\sqcup_e^\pm(d_M))\right). \end{equation}
Recall that $\heo\ep=\tc_\ep f_\ep=f_\ep\tc_\ep$, and dually $\perp_{\epc}=\heo{\nu^{-1}\epc}$.

\begin{theorem}\label{T:Lstab}
	For every rigid presentation $e$, the assignments $L_e^+$ and $R_{\ec}^+$ extend to functors, and there are natural isomorphisms
	\[ L_e^+\cong\heo\ep, \qquad R_{\ec}^+\cong\perp_{\epc}. \]
\end{theorem}
\begin{proof}
	We prove the first assertion. Let $F=F_{Q,\S}$, let $X$ correspond to $d_M$, and put
	$M'=F(X_e^+)=\coker(\wtd{\sqcup}_e^+(d_M))$.
	Applying $F$ to the \emph{categorical} approximation triangle \eqref{eq:sqtri+} gives
	\begin{equation}\label{eq:long}
	 F(U_X^+)\longrightarrow M'\longrightarrow M
	 \xrightarrow{g} F(U_X^+[1]).
	\end{equation}
	Here $F(U_X^+)\in\op{add}(E)$, while $F(U_a[1])$ is the module underlying
	$\tau e_a$ (including the injective module when $e_a$ is shifted projective).
	The defining universal map and fullness of $F$ show that the components of $g$
	span all maps $M\to F(U_a[1])$. Since
	$\check{\mc F}(\ep)=\op{Cogen}(\tau E)$, we have
	\[ \ker g=\bigcap_{\phi\in\Hom_A(M,\tau e)}\ker\phi =\tc_\ep(M). \]
	Thus \eqref{eq:long} gives $F(U_X^+)\rightarrow M'\rightarrow\tc_\ep(M)\rightarrow0$.
	Applying the right-exact functor $L_e$, which kills $F(U_X^+)$, yields
	\[L_e(M')\cong L_e(\tc_\ep(M)) =f_\ep\tc_\ep(M)=\heo\ep(M).\]
	Deleting $e$-summands does not change $L_e(M')$, so \eqref{eq:Lepm}
	identifies the left side with $L_e^+(M)$. The comparison is induced by the universal approximation and
	the torsion-radical inclusion. For a morphism of modules, lift it
	through the full functor $F$; functoriality of the approximation
	and of $\tc_\ep$ makes the resulting square commute.
	Since the induced map is $f_\ep\tc_\ep$ on morphisms, it is
	independent of the chosen lift and gives a natural isomorphism
	on $\rep(A)$.
	The proof for $R_\ec^+$ is dual.
\end{proof}

Let $\widehat Q$ be the full subquiver obtained by deleting a vertex set $I$. We restrict a projective presentation to $\widehat Q$ by deleting every summand $P_i$ with $i\in I$.

\begin{theorem}\label{T:L_e^-}
	Fix $\varepsilon\in\{+,-\}$ and assume that $e^\varepsilon$ is extended-reachable.
	Choose a mutation sequence $\mu_{e^\varepsilon}$ with $\mu_{e^\varepsilon}(e^\varepsilon)=\varepsilon(P_{I^c}\oplus P_I)$.
	For $\varepsilon=+$, omit the final global shift; this does not change the terminal QP.
	Then, canonically on isomorphism classes,
	\[ \L_{(e)}^\varepsilon \cong \L_{(\varepsilon P_I)}^\varepsilon\mu_{e^\varepsilon}. \]
	Equivalently, $\L_{(e)}^\varepsilon(d)$ is obtained by applying this mutation sequence to $d$ and restricting to the complementary vertices $I^c$.
	In particular, general presentations are sent to general presentations.
\end{theorem}
\begin{proof}
	Let $\Phi$ be the triangle equivalence attached to the chosen mutation sequence.
	Proposition \ref{P:reduction-adjoints} gives
	$\Phi_U\pi_U^\varepsilon \cong \pi_{\varepsilon P_I}^\varepsilon\bar\Phi$.
	Passing through Plamondon's correspondence yields
	\[ \mu_{e^\varepsilon}\sqcup_e^\varepsilon \cong \sqcup_{\varepsilon P_I}^\varepsilon\mu_{e^\varepsilon}. \]
	Proposition \ref{L:CY-Lcompare} identifies the decorated representation
of each compatible object in the reduced category with its normalized
projection $\L_e$. Using this on both sides of the preceding identity
and the suspension normalization of $\lambda_e^\varepsilon$ gives
	$\L_{(e)}^\varepsilon \cong \L_{(\varepsilon P_I)}^\varepsilon\mu_{e^\varepsilon}$.
	
	For the signed projective presentations, Lemma \ref{L:approxtrunc} identifies $\sqcup_{P_I[1]}^-$ with $\eo{P_I[1]}$ and $\sqcup_{P_I}^+$ with $\eor{P_I}$. Their approximation triangles change a presentation only by terms in $\op{add}(P_I)$ and maps incident with those terms; Corollary \ref{C:factor}, Remark \ref{r:Lneg}, and the dual statement therefore show that $\L_{(\varepsilon P_I)}^\varepsilon$ agrees with restriction to $I^c$ on projective presentations.
	Finally, mutation is generically defined on presentation spaces and full-subquiver restriction is a surjective linear map, so general presentations map to general presentations \cite{GLFSa,Fgr}.
\end{proof}

\begin{example}\label{E:stabnotgeneric}
The functor $\heo\ep$ may fail to preserve general presentations for other algebras.
Let $B=\k\langle x,y\rangle/(x,y)^2$ and let $A$ be the
seven-dimensional algebra with quiver $1\xrightarrow{a}2$, loops $x,y$
at vertex $2$, and relations $(x,y)^2=0$.  A general presentation
$d:P_2\to P_1^2$ of weight $\delta=(2,-1)$ has, after base changes,
the form $(1,z)^\top$, where $0\ne z\in\k x+\k y$.  For
$M_z=\coker d$, the arrow $a_{M_z}:\k^2\hookrightarrow B$ has image
$\k1+\k z$.

The presentation $e=(P_2\xrightarrow{a}P_1)$ of weight
$\ep=(1,-1)$ is indecomposable rigid, and
$\rep(e^\perp)=\{N:a_N\text{ is invertible}\}\simeq\rep(B)$.
Since $\k z$ is the largest $B$-submodule of $\k1+\k z$,
$\heo\ep(M_z)$ corresponds to the simple $S_B=B/(x,y)$.
But every general-presentation cokernel over the local algebra $B$
is free or zero, so $S_B$ is not general.  Thus rigid projection
need not preserve general presentations over arbitrary finite-dimensional
algebras.
\end{example}

\begin{corollary}\label{C:mulift}
	Fix $\varepsilon\in\{+,-\}$ and let $e$ be a basic rigid presentation such that $e^\varepsilon$ is extended-reachable. Let $(Q',\S')=\mu_k(Q,\S)$ and $e'=\mu_k(e)$.	
	Then $\mu_k$ induces the following mutation of the projected QP:
	\[
	\mu_k^{e,\varepsilon}=
	\begin{cases}
		\mathrm{id},&\ep(k)\neq0,\\
		\mu_j,&\ep(k)=0,
	\end{cases}
	\]
	where $j$ is the complementary index in
	Lemma~\ref{L:existj}. With compatible choices of
	the projected QPs, the diagram
	\[
	\xymatrix{
		\drep(Q,\S)
		\ar[r]^{\mu_k}
		\ar[d]_{\L_{(e)}^\varepsilon}
		&
		\drep(Q',\S')
		\ar[d]^{\L_{(e')}^\varepsilon}
		\\
		\drep({_e(Q,\S)}^\varepsilon)
		\ar[r]_{\mu_k^{e,\varepsilon}}
		&
		\drep({_{e'}(Q',\S')}^\varepsilon)
	}
	\]
	commutes on isomorphism classes. The same holds
	for mutation sequences by composition.
	
	Conversely, choose a mutation sequence
	$\mu_{e^\varepsilon}$ such that
	$\mu_{e^\varepsilon}(e^\varepsilon)
	=\varepsilon(P_{I^c}\oplus P_I)$ with
	$I^c=Q_0\setminus I$,
	where $I$ corresponds to the summands of $e$.
	Any mutation sequence $\mub^e$ of the projected QP
	that is also admissible on
	$\mu_{e^\varepsilon}(Q,\S)$ at the vertices of $I^c$
	is realized by $\mub^e\mu_{e^\varepsilon}$.
\end{corollary}

\begin{proof}
	We first consider a single mutation $\mu_k$.
	The compatibility of the $\E$-orthogonalizations
	with mutation, given by \eqref{eq:sqmu}, yields
	$\mu_k\sqcup_e^\varepsilon	\cong	\sqcup_{e'}^\varepsilon\mu_k$.
	By Proposition~\ref{L:CY-Lcompare},
	the normalized projection $\L_e$ agrees with
	the corresponding Calabi--Yau reduction
	on compatible presentations.
	
	Theorem~\ref{T:mucomp} describes the change
	of the $\varepsilon$-completion. If
	$\ep(k)\neq0$, then $(e')^\varepsilon=\mu_k(e^\varepsilon)$.
	The complementary summands correspond under
	mutation, so the projected QP is unchanged,
	with the chosen identifications.
	
	If $\ep(k)=0$, Theorem~\ref{T:mucomp} gives $(e')^\varepsilon=\sigma_j\mu_k(e^\varepsilon)$.
	Thus the two completions differ by an exchange
	of the $j$-th complementary summand.
	Under Calabi--Yau reduction, this is the
	exchange of the corresponding cluster-tilting
	summand. Since the completions are
	extended-reachable, the exchange is realized
	by mutation $\mu_j$ of the projected QP.
	
	In either case, the compatibility of the $\E$-orthogonalizations and the identification
	of normalized projection with reduction give
	$\L_{(e')}^\varepsilon\mu_k\cong \mu_k^{e,\varepsilon}\L_{(e)}^\varepsilon$.
	This proves the commuting diagram.
	The assertion for sequences follows by
	applying the same argument successively.
	
	For the converse, write $(\widehat Q,\widehat\S)=\mu_{e^\varepsilon}(Q,\S)$.
By definition, the projected QP is
$({\widehat Q},{\widehat\S})|_{I^c}$.
Let $\mub^e=\mu_{j_s}\cdots\mu_{j_1}$, with $j_a\in I^c$,
be a sequence of mutations which is admissible both for this
restriction and for $(\widehat Q,\widehat\S)$.
Mutation of a QP at a vertex of $I^c$ commutes with restriction
to $I^c$; hence
\[ \mub^e(\widehat Q,\widehat\S)|_J \simeq \mub^e\bigl({}_e(Q,\S)^\varepsilon\bigr). \]
Theorem~\ref{T:L_e^-} identifies $\L_{(e)}^\varepsilon$
with $\mu_{e^\varepsilon}$ followed by restriction of
presentations to $I^c$. Applying the corresponding mutation and
restriction identities at $j_1,\dots,j_s$ shows that the sequence
$\mub^e\mu_{e^\varepsilon}$ realizes the prescribed mutations
of the projected QP.
\end{proof}

\subsection{Application to $F$-polynomials}\label{ss:Fpolynomials}
The preceding results apply to facets of Newton polytopes of $F$-polynomials.
Recall from \cite{DWZ2} that the {\em $F$-polynomial} of a representation $M$ records the Euler characteristics of its representation Grassmannians:
\[ F_M(\b{y}) = \sum_{d} \chi(\Gr_d(M)) \b{y}^d. \]
Let $l_e$ be the equivalence $\rep(e^\perp) \xrightarrow{\sim} \rep({_eA})$, and
$\iota_\ep$ be the monomial change of variables:
\[ \iota_\ep: \b{y}_e^d \mapsto \b{y}^{d \C_{e^\perp}}. \]

\begin{corollary}\label{C:faces}
	Let $M\in\rep(Q,\S)$, let $\ep$ be the indivisible outer normal of a facet ${\sf\Lambda}$ of $\N(M)$, and let $e$ be a rigid presentation of weight $\ep$.
	Then the restriction of $F_M$ to ${\sf\Lambda}$ is
	\[ \b y^{\dv t_\ep(M)} \iota_\ep\!\left(F_{l_e(\heo\ep(M))}(\b y_e)\right). \]
	If $M$ is general of weight $\delta$ and $e^+$ is extended-reachable, then $\lambda_e^+(\heo\ep(M))$ is general over the explicit positive projected QP.
	Thus, up to a monomial factor and a monomial change of variables, the facet restriction of a generic $F$-polynomial is again generic.
\end{corollary}
\begin{proof}
	The restriction formula is \cite[Theorem 7.8]{Fcf}, whose map $\pi_\ep$ is $l_e\heo\ep$ in the present notation.
	If $d_M$ is a general presentation of $M$, Theorem \ref{T:Lstab} and Definition \ref{D:Lpm} give
	\[ \lambda_e^+(\heo\ep(M)) \cong\coker(\L_{(e)}^+(d_M)). \]
	Theorem \ref{T:L_e^-} realizes $\L_{(e)}^+$ as mutation followed by restriction, so the result is general.
\end{proof}

\begin{example}\label{E:tame-normal-face}
A straightforward generalization to tame normal vectors would fail.
Let $Q$ be the quiver \vspace{-.11in}
\[\xymatrix@R=2em@C=2em{
 & 2\ar[dl]_{a} & \\
 1\ar@/^.2pc/[rr]^{c}\ar@/_.2pc/[rr]_{d} && 3\ar[ul]_{b}
} \]
with potential $\mc{S}=acb$, with paths written in traversal order.
Its Jacobian algebra $A$ has relations $ac=cb=ba=0$ and dimension
$10$. Every formal potential with nonzero cubic part is
right-equivalent to $\mc{S}$: cyclically write it as $aFb$ and complete
the nonzero linear part of $F$ to an invertible change of the arrows
$c,d$.

Put $\ep=(1,0,-1)$ and
$R_\lambda=\coker(P_3\xrightarrow{c-\lambda d}P_1)$ for
$\lambda\ne0$. These modules have dimension $(1,0,1)$, and applying
$\Hom(-,R_\mu)$ to the indicated presentation gives
$\mu-\lambda:\k\to\k$. Thus $\e(\ep,\ep)=0$.
However, $\PHom(\ep)=\k c+\k d$ and
$\End(P_1)=\End(P_3)=\k$, so every orbit in this
presentation space has dimension at most one. Hence $\ep$ is tame
but not real.

A general presentation of weight $(1,1,-1)$ is
\[ P_3\xrightarrow{\binom{\alpha c+\beta d}{tad}}P_1\oplus P_2, \qquad \alpha\beta\ne0. \]
Since $a(\alpha c+\beta d)=\beta ad$, a triangular base change
kills the second component. Its cokernel is therefore
$M=P_2\oplus R_\lambda$ for a general nonzero $\lambda$.

The arrow $d:M_1\to M_3$ is invertible, so $s_1=s_3$ supports
$\N(M)$ with outer normal $\ep$. The composition series of the
uniserial $P_2$ is $S_2,S_1,S_3,S_2$ from top to socle, and the
only nonzero proper submodule of $R_\lambda$ is $S_3$.
Consequently, with $u=y_1y_3$ and $v=y_2$,
\[ \left.F_M\right|_{s_1=s_3} =(1+u)(1+v+uv+uv^2) =(1+u)(1+v)(1+uv). \]
Its Newton polygon is a centrally symmetric hexagon. The additional
$y_3$-term of $F_M$ shows that this face is a facet.

This hexagon cannot be a monomial translate and integral monomial
image of the Newton polygon of $F_N$ for any representation $N$
of a loopless, $2$-cycle-free two-vertex QP. Such a quiver is acyclic;
orient its arrows $1\to2$ and write $\dv N=(r,s)$. Its Newton polygon
lies in $[0,r]\times[0,s]$ and contains $(0,0),(0,s),(r,s)$.
Central symmetry would have center $(r/2,s/2)$ and would force
the fourth corner $(r,0)$ to occur. Thus every arrow map would
vanish and the polygon would be a rectangle. A monomial map with
two-dimensional image is injective on the exponent plane and
preserves central symmetry and the number of vertices. It therefore
cannot produce the displayed hexagon.

Thus the generic-facet conclusion of Corollary 9.11 in the manuscript
fails for a tame normal, even tropically and for a generic potential.
The source here is not $\E$-rigid; QPs with loops or $2$-cycles are
not covered by the obstruction.
\end{example}

\begin{remark}\label{R:arbitrary-faces}
(1). The facet assumption in Corollary \ref{C:faces} is not essential. Let ${\sf\Lambda}$ be any face of $\N(M)$ and choose an integral weight $\ep$ whose supporting hyperplane cuts out ${\sf\Lambda}$. If $\ep$ is the weight of a rigid presentation $e$, then the same formula holds with ${\sf\Lambda}$ in place of the facet. Under the same hypotheses on $M$ and $e^+$, the reduced representation is general.

(2). The generic cluster character of \cite{P} is
	\[ C_{\op{gen}}(\dtc)=\b x^{-\dtc}F_\dtc(\b y), \qquad y_i=\b x^{b_i}. \]
	Hence it differs from the generic $F$-polynomial only by a monomial factor and a monomial change of variables.
	If the extended $B$-matrix has full rank, the same conclusion as in Corollary \ref{C:faces} applies to the corresponding facet of the Newton polytope of $C_{\op{gen}}(\dtc)$.
\end{remark}

\begin{question}
Does the final assertion of Corollary \ref{C:faces} remain true if
the generic $F$-polynomial is replaced by the $F$-polynomial of a theta function \cite{GHKK}?	
\end{question}

\section*{Acknowledgment}
The author thanks Xiaoyue Lin for proofreading the manuscript.
The author would like to thank anonymous referees for helpful suggestions, including pointing out the references 
\cite{AHMV16ref}, \cite{KL23ref}, etc.

\printbibliography[title={References}]
\end{document}